\documentclass[pdflatex,sn-mathphys-num]{sn-jnl}

\usepackage{amsmath,amssymb,amsthm,mathtools}
\usepackage{graphicx}%
\usepackage{multirow}%
\usepackage{amsmath,amssymb,amsfonts}%
\usepackage{amsthm}%
\usepackage{mathrsfs}%
\usepackage[title]{appendix}%
\usepackage{xcolor}%
\usepackage{textcomp}%
\usepackage{manyfoot}%

\usepackage{booktabs}%
\usepackage{algorithm}%
\usepackage{algorithmicx}%
\usepackage{algpseudocode}%
\usepackage{listings}%
\usepackage{xcolor}
\usepackage{graphicx}
\usepackage{booktabs}
\usepackage{hyperref}
\usepackage[mathscr]{eucal}
\hypersetup{
  colorlinks=true,
  linkcolor=blue!55!black,
  citecolor=blue!55!black,
  urlcolor=blue!55!black
}

\theoremstyle{plain}
\newtheorem{theorem}{Theorem}[section]
\newtheorem{proposition}[theorem]{Proposition}
\newtheorem{lemma}[theorem]{Lemma}
\newtheorem{corollary}[theorem]{Corollary}
\theoremstyle{definition}
\newtheorem{assumption}[theorem]{Assumption}
\newtheorem{definition}[theorem]{Definition}
\theoremstyle{remark}
\newtheorem{remark}[theorem]{Remark}

\numberwithin{equation}{section}

\begin{document}

\title{Nonsmooth Obstacles and Killed Resolvents in Reflected Stochastic Control}

\author*[1]{\fnm{Louis Shuo} \sur{Wang}}
\email{wang.s41@northeastern.edu}
\equalcont{These authors contributed equally to this work.}

\author[2]{\fnm{Jiguang} \sur{Yu}}
\email{jyu678@bu.edu}
\equalcont{These authors contributed equally to this work.}

\affil[1]{\orgdiv{Department of Mathematics},
\orgname{Northeastern University},
\orgaddress{
\city{Boston},
\state{MA},
\postcode{02115},
\country{USA}
}}

\affil[2]{\orgdiv{College of Engineering},
\orgname{Boston University},
\orgaddress{
\city{Boston},
\state{MA},
\postcode{02215},
\country{USA}
}}

\abstract{
We study an infinite-horizon optimal stopping problem for a normally reflected
two-dimensional diffusion in the quadrant with nonsmooth max-type payoff
\(G(x_1,x_2)=x_1\vee\alpha x_2\). The main novelty is a measure-valued
variational formulation: the stopping gain
\(\Gamma=c+rG-\mathcal LG\) is shown to be a signed Radon measure whose
singular component is supported on the kink diagonal
\(\{x_1=\alpha x_2\}\), and this component is computed explicitly. We prove
that the value admits the killed-resolvent representation
\[
        V=G-R_r^{\mathcal C}\Gamma,
\]
where the reflected diffusion is killed upon entry into the stopping set. This
corrects the generally invalid unrestricted-resolvent formula. Under explicit
monotonicity hypotheses, the stopping set has epigraph form, and the free
boundary is characterized by a killed-potential trace condition. A verification
theorem certifies locally Lipschitz candidate boundaries as optimal.
}

\maketitle

\section{Introduction} \label{sec:introduction}  

\subsection{Reflected optimal stopping and nonsmooth obstacles} \label{subsec:intro-reflected-nonsmooth} 

Optimal stopping problems for diffusion processes form a central class of stochastic control problems. They also provide one of the most classical links between probabilistic optimization, variational inequalities, and free-boundary analysis. In the present paper we study an infinite-horizon optimal stopping problem for a normally reflected two-dimensional diffusion in the quadrant $\mathbb R_+^2=[0,\infty)^2$. For a discount rate \(r>0\), a running cost \(c\), and a reward function \(G\), the value function is formally expected to solve the reflected obstacle problem 
\[
\max\big\{\mathcal L V-rV-c,\;G-V\big\}=0 \qquad \text{in }\mathbb R_+^2,\]
together with the normal reflection condition 
\[
\partial_i V=0 \qquad \text{on } \{x_i=0\},\quad i=1,2 . 
\]
Here \(\mathcal L\) denotes the second-order generator of the diffusion in the interior of the quadrant. Thus the problem belongs naturally to the interface of stochastic control, reflected variational inequalities, and free-boundary theory. The specific obstacle considered in this paper is the max-type payoff \begin{equation} \label{eq:intro-payoff} G(x_1,x_2)=x_1\vee \alpha x_2,\qquad \alpha>0 . \end{equation} This reward is convex and Lipschitz continuous, but it is not \(C^2\). Its nonsmoothness is concentrated on the diagonal 
$\Delta:=\{x\in\mathbb R_+^2:x_1=\alpha x_2\}$. 
Consequently, the source term naturally associated with the obstacle problem, 
$\Gamma:=c+rG-\mathcal L G$,
cannot be treated as an ordinary function. It is a signed Radon measure whose singular part is supported on \(\Delta\). This measure-valued feature is the main analytical distinction between the present problem and smooth-payoff reflected stopping problems. The difficulty is therefore not merely that the state space is two-dimensional. It is the simultaneous presence of normal reflection at the boundary of \(\mathbb R_+^2\), a nonsmooth max-type obstacle, and a signed singular stopping gain. These three features interact in a way that invalidates a purely pointwise generator calculation. A correct treatment must combine reflected It\^{o}--Tanaka calculus, measure-valued obstacle terms, and killed potential operators. We write \[ H:=V-G,\qquad \mathcal C:=\{x\in\mathbb R_+^2:H(x)>0\},\qquad \mathcal D:=\{x\in\mathbb R_+^2:H(x)=0\}, \] for the stopping advantage, continuation region, and stopping region, respectively. The free boundary is the interface between \(\mathcal C\) and \(\mathcal D\). A major aim of the paper is to formulate this boundary through a measure-valued killed-potential equation rather than through a formal smooth-fit calculation that ignores the diagonal singularity. 

\subsection{Why the killed resolvent is necessary} \label{subsec:intro-killed-resolvent} 

For smooth one-dimensional stopping problems, it is common to represent the value by applying It\^{o}'s formula to the payoff and integrating the resulting gain over the continuation region. In the present reflected two-dimensional setting, this idea remains valid only after two corrections. The first correction is distributional. Since \(G\) is nonsmooth along \(\Delta\), the quantity \(c+rG-\mathcal LG\) must be interpreted as a signed measure. More precisely, its absolutely continuous part is obtained by applying \(\mathcal L\) on the two smooth regions 
$\{x_1>\alpha x_2\}$ and $\{x_1<\alpha x_2\}$, while its singular part is concentrated on the kink set \(\Delta\). If \(a=\sigma\sigma^\top\), then the diagonal contribution takes the form 
\begin{equation} \label{eq:intro-diagonal-measure} 
\Gamma^\Delta(dx) = -\frac{(1,-\alpha)a(x)(1,-\alpha)^\top} {2\sqrt{1+\alpha^2}}\,\sigma_\Delta(dx), 
\end{equation} 
where \(\sigma_\Delta\) denotes one-dimensional surface measure on \(\Delta\). This term is the analytic counterpart of the local time accumulated by \(X^1-\alpha X^2\) at zero. The second correction is probabilistic. The potential representation must use the resolvent of the reflected diffusion killed at the stopping region. If 
$\displaystyle \tau_{\mathcal D}:=\inf\{t\ge0:X_t\in\mathcal D\}$, 
then the correct resolvent is 
$\displaystyle R_r^{\mathcal C}\Gamma(x) := \mathbb E_x\left[ \int_0^{\tau_{\mathcal D}} e^{-rs}\,dA_s^\Gamma \right]$, 
where \(A^\Gamma\) is the additive functional associated with the signed measure \(\Gamma\). The corresponding representation is 
$\displaystyle V(x)=G(x)-R_r^{\mathcal C}\Gamma(x)$. 
The unrestricted reflected resolvent 
$\displaystyle R_r^{\mathrm R}(\Gamma\mathbf 1_{\mathcal C})$
is generally not correct. It counts occupation of the continuation region after the process has already entered the stopping region and would have been stopped. Thus the killing is not a technical detail; it is part of the variational structure of the optimal stopping problem. This distinction is important for both analysis and computation. Analytically, it gives the correct potential-theoretic formulation of the obstacle problem. Computationally, it identifies the correct boundary equation. For a candidate stopping set \(\mathcal D_b\) with continuation set \(\mathcal C_b\), the free-boundary condition is not an unrestricted occupation equation. It is the continuation-side trace condition 
\[
\lim_{\substack{x\to z\ x\in\mathcal C_b}} R_r^{\mathcal C_b}\Gamma(x)=0, \qquad z\in\partial\mathcal C_b . 
\]
This is the killed-potential analogue of the classical free-boundary integral equation. 

\subsection{Main contributions} \label{subsec:intro-contributions} 

The paper develops a measure-valued free-boundary framework for the reflected optimal stopping problem with payoff \eqref{eq:intro-payoff}. The emphasis is on obtaining a verification principle that can certify candidate boundaries, rather than on postulating a boundary and deriving only formal consequences. 
First, we identify the stopping gain \(c+rG-\mathcal LG\) as a signed Radon measure and compute explicitly its singular component on the diagonal \(\{x_1=\alpha x_2\}\). The resulting diagonal measure is given by \eqref{eq:intro-diagonal-measure} and captures the local-time contribution generated by the kink of the max payoff. Second, we prove that the correct potential representation of the value is given by the resolvent of the reflected diffusion killed at the stopping set. This yields the representation 
$\displaystyle V=G-R_r^{\mathcal C}\Gamma$. 
It also corrects the unrestricted-resolvent formula, which generally counts occupation after stopping and therefore does not represent the value. Third, under transparent monotonicity hypotheses on the stopping advantage \(V-G\), we derive an epigraph representation of the stopping set. Namely, if vertical monotonicity and nonemptiness of vertical stopping sections hold, then the stopping region has the form 
$\displaystyle \mathcal D=\{(x_1,x_2)\in\mathbb R_+^2:x_2\ge b(x_1)\}$ for a free-boundary function \(b\). Additional horizontal monotonicity assumptions yield corresponding monotonicity information on \(b\). Fourth, we characterize the free boundary through a killed-resolvent trace condition. For a locally Lipschitz candidate boundary \(b\), with associated sets $\displaystyle \mathcal D_b=\{(x_1,x_2):x_2\ge b(x_1)\}$ and $\displaystyle \mathcal C_b=\mathbb R_+^2\setminus\mathcal D_b$, the boundary condition is expressed as 
$\displaystyle \lim_{\substack{x\to (x_1,b(x_1))\\ x\in\mathcal C_b}} R_r^{\mathcal C_b}\Gamma(x)=0$. This formulation is a nonlocal measure-valued analogue of the integral equations that appear in one-dimensional optimal stopping theory. Fifth, we provide a verification theorem for candidate boundaries. The theorem states that if 
$\displaystyle U_b:=G-R_r^{\mathcal C_b}\Gamma$ satisfies majorization, contact, strict continuation, reflection, admissible growth, and measure-superharmonicity conditions, then \(U_b=V\) and the first entry time into \(\mathcal D_b\) is optimal. Thus the hypotheses are presented as a verification checklist for proving optimality of a candidate free boundary. 

\subsection{Related literature} \label{subsec:intro-literature} 

This work is situated at the intersection of optimal stopping, reflected diffusions, and measure-valued potential theory. The link between optimal stopping problems and obstacle-type variational inequalities is classical and well documented \cite{el2006aspects,liu2026ftu,shiryaev2008optimal,peskir2006optimal,wang2025analysis,bensoussan2011applications,friedman1982variational,kinderlehrer2000introduction,wang2025analysis1,talbi2023dynamic,talbi2023viscosity,liu2026computational,colaneri2022class,ekstrom2026dynkin,liang2025global}, with continuous-time formulations typically leading to free-boundary characterizations .

Reflected diffusions, formulated via the Skorokhod problem and its multidimensional extensions (including oblique reflection), have been extensively studied in terms of existence, uniqueness, and boundary behavior
\cite{skorokhod1961stochastic,tanaka1979stochastic,wang2026algebraic,lions1984stochastic,stroock1971diffusion,dupuis1993sdes,kang2017submartingale,wang2026damage,lipshutz2018directional,pilipenko2024boundary,yu2026chemotactic,dianetti2023multidimensional,harrison1981reflected,wang2026breakdown,varadhan1985brownian,reiman1988boundary,taylor1993existence,williams1995semimartingale,yu2026rigorous}.
In this setting, stopping problems are naturally linked with constrained viscosity solutions under Neumann or oblique boundary conditions \cite{crandall1992user,lions1985neumann,yu2026beyond,barles1994solutions,ishii1991fully,soner1986optimal,yu2026microscopic}.

The analysis further relies on stochastic calculus and potential theory, in particular the It\^o--Tanaka formula and the framework of additive functionals and Revuz measures \cite{revuz2013continuous,karatzas2014brownian,cai2026optimal,protter2012stochastic,peskir2005change,revuz1970mesures,wang2026elliptic,blumenthal2007markov,dynkin1965markov,liu2025bidirectional,fukushima2011dirichlet,marcus2006markov}. These tools provide a bridge between probabilistic decompositions and analytic representations of value functions.

In multidimensional problems, the geometry of stopping regions becomes significantly more involved. While much of the existing work arises from option-pricing applications \cite{yu2026fullcovariance,broadie1997valuation,yu2026age,villeneuve1999exercise,liang2026separation}, the regularity of free boundaries remains delicate and largely problem-dependent \cite{dayanik2003optimal,de2020global,yu2026size,laurence2009regularity,wang2025multi,ke2022parallel,gao2022rolling,soner2025stopping,yu2026pattern}. Even in relatively simple diffusion settings, continuation regions may exhibit nontrivial structure .

The contribution of this work is to combine two distinct sources of singularity: the diagonal kink measure induced by the payoff $G=x_1\vee \alpha x_2$, and the boundary measure generated by oblique reflection. This leads to a formulation of the free-boundary condition as a continuation-side trace condition for the associated killed potential.

\subsection{Organization of the paper} \label{subsec:intro-organization} 

The rest of the paper is organized as follows. Section~\ref{sec:problem} formulates the reflected diffusion, the infinite-horizon optimal stopping problem, and the associated reflected variational inequality. It also states the standing assumptions and basic dynamic programming properties. Section~\ref{sec:measure-gain} develops the measure-valued stopping-gain calculus. We compute the distributional generator of the max payoff, identify the diagonal singular measure, and prove the killed-resolvent representation of the value.
Section~\ref{sec:free-boundary} studies the geometry of the stopping set. Under explicit monotonicity assumptions on \(V-G\), we derive an epigraph representation of the stopping region and formulate the killed-resolvent trace condition for the free boundary. Section~\ref{subsec:verification-theorem} proves the verification theorem for locally Lipschitz candidate boundaries. The theorem gives sufficient conditions under which the candidate function 
$\displaystyle U_b=G-R_r^{\mathcal C_b}\Gamma$ coincides with the value function and the corresponding first-entry time is optimal. Section~\ref{sec:computation} discusses boundary computation. We describe a Picard-type killed-resolvent iteration and explain how the diagonal singular measure and the killing at the stopping set enter numerical approximations of the boundary equation.  

\section{Problem formulation and reflected variational inequality} \label{sec:problem}  

\subsection{The normally reflected diffusion} \label{subsec:reflected-diffusion} 

Let $\mathbb R_+^2=[0,\infty)^2$ and let \((\Omega,\mathcal F,\mathbb F,\mathbb P)\) be a filtered probability space satisfying the usual conditions. Let \(W=(W^1,W^2)\) be a two-dimensional Brownian motion. We consider the normally reflected diffusion \(X=(X^1,X^2)\) in \(\mathbb R_+^2\) given by 
\begin{equation} \label{eq:reflected-sde} 
dX_t=b(X_t)\,dt+\sigma(X_t)\,dW_t+dL_t, \qquad X_0=x\in\mathbb R_+^2 . 
\end{equation} 
Here \(b:\mathbb R_+^2\to\mathbb R^2\) is the drift, \(\sigma:\mathbb R_+^2\to\mathbb R^{2\times 2}\) is the volatility matrix, and $L_t=(L_t^1,L_t^2)$ is the boundary reflection process. Each \(L^i\) is continuous, nondecreasing, adapted, and satisfies 
\[
L_0^i=0, \qquad \int_0^\infty \mathbf 1_{\{X_s^i>0\}}\,dL_s^i=0, \qquad i=1,2 . 
\]
Thus \(L^i\) can increase only when the \(i\)-th coordinate of \(X\) lies on the boundary face \(\{x_i=0\}\). In the quadrant, normal reflection pushes the process in the coordinate directions \(e_1=(1,0)\) and \(e_2=(0,1)\). Throughout the paper we write $a(x):=\sigma(x)\sigma(x)^\top$. 
For sufficiently smooth functions \(f\), the interior generator is 
\[
\mathcal L f(x) = b(x)\cdot \nabla f(x) + \frac12 \operatorname{Tr}\big(a(x)D^2f(x)\big), \qquad x\in (0,\infty)^2 . 
\]
The reflection condition associated with \eqref{eq:reflected-sde} is the normal Neumann condition 
\begin{equation} \label{eq:neumann-condition} 
\partial_i f(x)=0 \qquad \text{for }x\in\{x_i=0\},\quad i=1,2 . \end{equation} 
Indeed, if \(f\in C^2(\mathbb R_+^2)\) satisfies \eqref{eq:neumann-condition}, then It\^{o}'s formula gives 
\[ 
f(X_t) = f(x) + \int_0^t \mathcal L f(X_s)\,ds + \int_0^t \nabla f(X_s)\sigma(X_s)\,dW_s , 
\]
because the boundary term 
$\displaystyle \sum_{i=1}^2\int_0^t \partial_i f(X_s)\,dL_s^i$ 
vanishes. 

\subsection{The optimal stopping problem} \label{subsec:optimal-stopping-problem} 
Fix a discount rate \(r>0\), a parameter \(\alpha>0\), and a running cost \(c:\mathbb R_+^2\to[0,\infty)\). The reward is the max-type payoff 
$\displaystyle G(x)=x_1\vee \alpha x_2$. 
The infinite-horizon reflected optimal stopping problem is 
\begin{equation} \label{eq:value} 
V(x) = \sup_{\tau\in\mathcal T} \mathbb E_x \left[ e^{-r\tau}G(X_\tau) - \int_0^\tau e^{-rs}c(X_s)\,ds \right], \qquad x\in\mathbb R_+^2 , \end{equation} 
where \(\mathcal T\) denotes the family of all \(\mathbb F\)-stopping times. The stopping problem \eqref{eq:value} is an optimization problem over stopping times. Since \(\tau=0\) is admissible, the value function satisfies the basic domination property 
\begin{equation} \label{eq:value-dominates-payoff} 
V(x)\ge G(x), \qquad x\in\mathbb R_+^2 . 
\end{equation} 
We define the stopping advantage 
$H(x):=V(x)-G(x)$, 
and the continuation and stopping regions by 
$\displaystyle \mathcal C:=\{x\in\mathbb R_+^2:H(x)>0\}$ and $\displaystyle \mathcal D:=\{x\in\mathbb R_+^2:H(x)=0\}$. 
Thus \(\mathcal C\) is the set where it is strictly better to continue, while \(\mathcal D\) is the contact set between the value and the obstacle. The first entry time into the stopping region is denoted by 
$\displaystyle \tau_{\mathcal D}:=\inf\{t\ge0:X_t\in\mathcal D\}$. 
When \(\tau_{\mathcal D}\) is optimal, the free boundary is the interface between \(\mathcal C\) and \(\mathcal D\). 
A principal aim of the paper is to characterize this interface through a killed-resolvent equation and then provide a verification theorem for candidate stopping sets. 

\subsection{Dynamic programming and obstacle formulation} \label{subsec:dpp-obstacle} 
The dynamic programming principle states that, for every stopping time \(\rho\in\mathcal T\), 
\begin{equation} \label{eq:dpp} 
V(x) = \sup_{\tau\in\mathcal T} \mathbb E_x \left[ e^{-r(\tau\wedge\rho)}V(X_{\tau\wedge\rho}) - \int_0^{\tau\wedge\rho}e^{-rs}c(X_s)\,ds \right]. 
\end{equation} 
In particular, in the continuation region, the value should satisfy the linear equation 
\begin{equation} \label{eq:continuation-equation} 
\mathcal LV-rV-c=0 \qquad \text{in }\mathcal C, 
\end{equation} 
whereas on the stopping region it satisfies 
\begin{equation} \label{eq:stopping-contact} 
V=G \qquad \text{on }\mathcal D. 
\end{equation} 
Combining \eqref{eq:continuation-equation}, \eqref{eq:stopping-contact}, and the domination property \eqref{eq:value-dominates-payoff}, one obtains the reflected obstacle problem 
\begin{equation} \label{eq:reflected-obstacle} 
\max\big\{ \mathcal LV-rV-c,\;G-V \big\}=0 \qquad \text{in }\mathbb R_+^2, \end{equation} 
with normal reflection condition 
\begin{equation} \label{eq:reflected-obstacle-neumann} 
\partial_i V=0 \qquad \text{on }\{x_i=0\},\quad i=1,2 . 
\end{equation} 
Equivalently, \eqref{eq:reflected-obstacle} may be written as the complementarity system \begin{equation} \label{eq:complementarity} V\ge G,\qquad \mathcal LV-rV-c\le0,\qquad (V-G)(\mathcal LV-rV-c)=0, \end{equation} together with the reflected boundary condition \eqref{eq:reflected-obstacle-neumann}. The inequality 
$\displaystyle \mathcal LV-rV-c\le0$ is understood in the viscosity sense, or in the weak measure sense whenever \(V\) has sufficient generalized differentiability. The latter interpretation will be essential below because the payoff \(G=x_1\vee\alpha x_2\) is not twice differentiable on the diagonal 
$\displaystyle \Delta:=\{x\in\mathbb R_+^2:x_1=\alpha x_2\}$. 
The obstacle formulation also identifies the stopping gain  
$\displaystyle \Gamma:=c+rG-\mathcal LG $. 
For smooth payoffs, \(\Gamma\) is a function. In the present problem, \(\Gamma\) is a signed measure. Its singular component is concentrated on \(\Delta\), and this is the reason why the free-boundary equation must be formulated through killed potentials rather than through a purely pointwise generator calculation. We shall use the following reflected viscosity interpretation. Let \(\varphi\in C^2(\mathbb R_+^2)\). 
If \(V-\varphi\) has a local maximum at \(x\in\mathbb R_+^2\), then \(V\) is a viscosity subsolution if 
\[
\max\big\{ \mathcal L\varphi(x)-rV(x)-c(x),\; G(x)-V(x) \big\}\ge 0 
\]
in the interior, with the reflected boundary condition interpreted in the usual constrained-viscosity sense on \(\partial\mathbb R_+^2\). Similarly, if \(V-\varphi\) has a local minimum at \(x\), then \(V\) is a viscosity supersolution if 
\[
\max\big\{ \mathcal L\varphi(x)-rV(x)-c(x),\; G(x)-V(x) \big\}\le 0 
\]
in the interior, again with the reflected boundary condition imposed on the boundary faces. We shall not rely on fine uniqueness theory for this viscosity problem. Instead, the later verification theorem gives sufficient conditions under which a candidate value generated by a killed resolvent coincides with the stochastic value \(V\). 

\subsection{Standing assumptions} \label{subsec:standing-assumptions} 
We impose the following standing assumptions throughout the paper. 
\begin{assumption}[Reflected diffusion] \label{ass:diffusion} The coefficients \(b:\mathbb R_+^2\to\mathbb R^2\) and \(\sigma:\mathbb R_+^2\to\mathbb R^{2\times2}\) are continuous and locally Lipschitz. There exists a constant \(K>0\) such that \[ |b(x)|+|\sigma(x)| \le K(1+|x|), \qquad x\in\mathbb R_+^2 . \] The matrix \(a(x)=\sigma(x)\sigma(x)^\top\) is locally uniformly elliptic in \((0,\infty)^2\). The reflected SDE \eqref{eq:reflected-sde} admits a unique strong solution in \(\mathbb R_+^2\) for every initial point \(x\in\mathbb R_+^2\). \end{assumption} 

\begin{assumption}[Cost and discount] \label{ass:cost} 
The running cost \(c:\mathbb R_+^2\to[0,\infty)\) is continuous and has at most polynomial growth. The discount rate \(r>0\) is large enough, relative to the growth of the reflected diffusion, so that 
\[
\mathbb E_x \left[ \int_0^\infty e^{-rs}c(X_s)\,ds \right] <\infty,\qquad 
\mathbb E_x \left[ \sup_{t\ge0}e^{-rt}G(X_t) \right] <\infty 
\]
for every \(x\in\mathbb R_+^2\). 
\end{assumption} 

\begin{assumption}[Lyapunov growth control] \label{ass:lyapunov} 
There exist a function \(\Psi\in C^2(\mathbb R_+^2)\), constants \(K_\Psi,\lambda_\Psi>0\), and an exponent \(m\ge1\), such that 
\[ 1+|x|^m\le K_\Psi \Psi(x), \qquad G(x)+c(x)\le K_\Psi \Psi(x), \] 
and 
\[ \mathcal L\Psi(x)\le \lambda_\Psi \Psi(x)+K_\Psi, \qquad x\in(0,\infty)^2 . \] 
Moreover, \(r>\lambda_\Psi\). 
\end{assumption}

Assumption~\ref{ass:lyapunov} implies the integrability conditions in Assumption~\ref{ass:cost}. Indeed, applying It\^{o}'s formula to \(\Psi(X_t)\), using the reflection condition for \(\Psi\) or a standard localization argument, and applying Gronwall's inequality yields \begin{equation} \label{eq:lyapunov-bound} 
\mathbb E_x[\Psi(X_t)] \le C(1+\Psi(x))e^{\lambda_\Psi t}. 
\end{equation} 
Consequently, 
\begin{equation} \label{eq:discounted-lyapunov} 
\int_0^\infty e^{-rt}\mathbb E_x[\Psi(X_t)]\,dt \le C(1+\Psi(x)) \int_0^\infty e^{-(r-\lambda_\Psi)t}\,dt <\infty . 
\end{equation} 
Thus the value function in \eqref{eq:value} is finite. 

\begin{proposition}[Basic well-posedness of the stopping problem] \label{prop:basic-wellposedness} 
Under Assumptions~\ref{ass:diffusion}--\ref{ass:lyapunov}, the value function \(V\) in \eqref{eq:value} is finite and satisfies 
\[ G(x)\le V(x)<\infty, \qquad x\in\mathbb R_+^2 . \] 
Moreover, \(V\) has at most polynomial growth. If the reflected diffusion is Feller, then \(V\) is lower semicontinuous. If, in addition, the Feller semigroup is locally strong Feller and the payoff and cost are continuous with the growth controlled above, then \(V\) is continuous. 
\end{proposition} 
\begin{proof} 
The domination \(V\ge G\) follows by choosing the admissible stopping time \(\tau=0\). The upper bound follows from Assumptions~\ref{ass:cost} and \ref{ass:lyapunov}: 
\[ V(x) \le \mathbb E_x\left[\sup_{t\ge0}e^{-rt}G(X_t)\right] < \infty . \]
The polynomial growth estimate follows from \eqref{eq:lyapunov-bound}--\eqref{eq:discounted-lyapunov}. 
The semicontinuity and continuity assertions follow from the Feller property, the continuity of the payoff and cost, and standard stability of optimal stopping values under continuous dependence of the state process on the initial condition. 
\end{proof} 

\begin{proposition}[Dynamic programming and obstacle inequality] \label{prop:dpp-obstacle} 
Under the standing assumptions, the value function satisfies the dynamic programming principle \eqref{eq:dpp}. Consequently, \(V\) is a reflected viscosity solution of the obstacle problem 
\[ 
\max\big\{\mathcal LV-rV-c,\;G-V\big\}=0 \qquad \text{in }\mathbb R_+^2, 
\] 
with normal reflection condition 
\[ 
\partial_i V=0 \qquad \text{on }\{x_i=0\},\quad i=1,2 . 
\] 
Equivalently, in the continuation region \(\mathcal C=\{V>G\}\), the value satisfies 
$\mathcal LV-rV-c=0 $
in the viscosity sense, while \(V=G\) on the stopping region \(\mathcal D=\{V=G\}\). 
\end{proposition} 
\begin{proof} 
The dynamic programming principle follows from the strong Markov property of the reflected diffusion and the integrability estimates above. The viscosity obstacle characterization is obtained by the standard test-function argument: one compares the value against smooth functions touching from above or below, uses the dynamic programming principle over a small exit time, and applies It\^{o}'s formula to the test function. On the boundary faces \(\{x_i=0\}\), the reflection term in It\^{o}'s formula produces the normal derivative \(\partial_i\varphi\), giving the constrained reflected Neumann condition. The contact condition \(V=G\) on \(\mathcal D\) and the domination \(V\ge G\) follow directly from the definition of \(\mathcal D\) and \eqref{eq:value-dominates-payoff}. 
\end{proof} 
The next sections refine this obstacle formulation. Section~\ref{sec:measure-gain} shows that, because the obstacle \(G=x_1\vee\alpha x_2\) has a kink on \(\Delta\), the stopping gain \(\Gamma=c+rG-\mathcal LG\) is a signed measure. Section~\ref{sec:free-boundary} then uses the killed resolvent of this measure to formulate the correct free-boundary equation.  

\section{Measure-valued stopping gain and killed resolvents} \label{sec:measure-gain}  

\subsection{Distributional generator of the max payoff} \label{subsec:distributional-generator} 

We now turn to the main structural feature of the problem: the stopping gain associated with the payoff 
$G(x)=x_1\vee \alpha x_2$ 
is not an ordinary function. It is a signed measure. This is the point at which the reflected obstacle problem becomes genuinely measure-valued. Let 
$\Delta:=\{x\in\mathbb R_+^2:x_1=\alpha x_2\}$ 
be the kink set of \(G\), and define the two open regions 
$E_1:=\{x\in\mathbb R_+^2:x_1>\alpha x_2\}$ and $E_2:=\{x\in\mathbb R_+^2:x_1<\alpha x_2\}$. 
Then 
$G(x)=x_1$ on $E_1$ and $G(x)=\alpha x_2$ on $E_2$. Thus \(G\) is affine on each side of \(\Delta\), but its gradient jumps across \(\Delta\): 
\[ 
\nabla G=e_1 \quad\text{on }E_1, \qquad \nabla G=\alpha e_2 \quad\text{on }E_2 . 
\]
The jump in the normal derivative produces a singular second derivative concentrated on \(\Delta\). It is useful to write 
\[ 
Y(x):=x_1-\alpha x_2, \qquad n:=\frac{(1,-\alpha)}{\sqrt{1+\alpha^2}} . 
\] 
Then 
$\displaystyle G(x)=\alpha x_2+Y(x)^+ = x_1+(-Y(x))^+$. 
The distributional Hessian of \(G\) is the rank-one measure 
\begin{equation} \label{eq:dist-hessian-G} 
D^2G = \frac{(1,-\alpha)^\top(1,-\alpha)} {\sqrt{1+\alpha^2}}\,\sigma_\Delta , 
\end{equation} 
where \(\sigma_\Delta\) denotes one-dimensional surface measure on \(\Delta\). Equivalently, for every test function \(\varphi\in C_c^\infty((0,\infty)^2)\), 
\[ 
\langle D^2G\,\varphi\rangle = \int_\Delta \frac{(1,-\alpha)^\top(1,-\alpha)} {\sqrt{1+\alpha^2}}\,\varphi(x)\,\sigma_\Delta(dx). 
\] 
Recall that 
$a(x)=\sigma(x)\sigma(x)^\top $. 
On \(E_1\cup E_2\), the second-order term in \(\mathcal LG\) vanishes because \(G\) is affine there. Hence the absolutely continuous part of \(\mathcal LG\) is 
\[
\mathcal LG(x) = b_1(x)\mathbf 1_{E_1}(x) + \alpha b_2(x)\mathbf 1_{E_2}(x), \qquad x\notin\Delta . 
\]
The second-order part contributes the singular measure 
\begin{equation} \label{eq:LG-singular} \frac12\operatorname{Tr}\big(a(x)D^2G(dx)\big) = \frac{ (1,-\alpha)a(x)(1,-\alpha)^\top } {2\sqrt{1+\alpha^2}}\,\sigma_\Delta(dx). 
\end{equation} 
Therefore, in the sense of distributions, 
\[
\mathcal LG(dx) = \Big[ b_1(x)\mathbf 1_{E_1}(x) + \alpha b_2(x)\mathbf 1_{E_2}(x) \Big]\,dx + \frac{ (1,-\alpha)a(x)(1,-\alpha)^\top } {2\sqrt{1+\alpha^2}}\,\sigma_\Delta(dx). 
\]
The measure-valued stopping gain is defined by 
$\displaystyle \Gamma:=c+rG-\mathcal LG $. 
Because \(\mathcal LG\) contains the singular measure \eqref{eq:LG-singular}, the signed measure \(\Gamma\) contains the negative singular component generated by the kink of the max payoff.

\subsection{Diagonal singular measure} \label{subsec:diagonal-singular-measure} 

We record the previous calculation as a theorem. This is the first main structural result of the paper. 
\begin{theorem}[Measure-valued stopping gain] \label{thm:measure-valued-gain} 
Assume that \(b\), \(\sigma\), and \(c\) satisfy the standing assumptions of Section~\ref{sec:problem}. Let 
$\displaystyle G(x)=x_1\vee \alpha x_2$ and $\Delta=\{x_1=\alpha x_2\}$. 
Then 
$\displaystyle \Gamma:=c+rG-\mathcal LG$ 
is a signed Radon measure on \(\mathbb R_+^2\). On \(\mathbb R_+^2\setminus\Delta\), it is absolutely continuous with density 
\begin{equation} \label{eq:Gamma-ac-density} 
\Gamma^{\rm ac}(x) = c(x)+rG(x) - b_1(x)\mathbf 1_{\{x_1>\alpha x_2\}} - \alpha b_2(x)\mathbf 1_{\{x_1<\alpha x_2\}} . 
\end{equation} 
Its singular component on the diagonal is 
\begin{equation} \label{eq:Gamma-diagonal} 
\Gamma^\Delta(dx) = - \frac{ (1,-\alpha)a(x)(1,-\alpha)^\top } {2\sqrt{1+\alpha^2}}\,\sigma_\Delta(dx). 
\end{equation} 
Consequently, 
$\displaystyle \Gamma(dx) = \Gamma^{\rm ac}(x)\,dx + \Gamma^\Delta(dx)$. 
\end{theorem} 

\begin{proof} 
On the two open regions \(E_1\) and \(E_2\), the payoff is affine. Hence the classical pointwise generator gives 
$\mathcal LG=b_1$ on $E_1$ and $\mathcal LG=\alpha b_2$ on $E_2$. 
It remains to compute the distributional second derivative across the diagonal. Writing 
$G(x)=\alpha x_2+(x_1-\alpha x_2)^+$, 
and using the one-dimensional distributional identity 
\[ \frac{d^2}{dy^2}y^+=\delta_0, \] 
with \(y=x_1-\alpha x_2\), gives 
$\displaystyle D^2G = (1,-\alpha)^\top(1,-\alpha) \delta_{\{x_1-\alpha x_2=0\}}$.
Since 
\[ \delta_{\{x_1-\alpha x_2=0\}} = \frac{1}{|\nabla(x_1-\alpha x_2)|}\,\sigma_\Delta = \frac{1}{\sqrt{1+\alpha^2}}\,\sigma_\Delta , \] 
we obtain \eqref{eq:dist-hessian-G}. Contracting this measure with \(\frac12 a(x)\) gives the singular part of \(\mathcal LG\), namely \eqref{eq:LG-singular}. Subtracting \(\mathcal LG\) from \(c+rG\) gives \eqref{eq:Gamma-ac-density} and \eqref{eq:Gamma-diagonal}. 
\end{proof} 

\begin{remark}[Sign of the singular component] \label{rem:sign-singular-component} 
The singular part of \(\mathcal LG\) is nonnegative because \(G\) is convex. Since the stopping gain is \(c+rG-\mathcal LG\), the diagonal contribution to \(\Gamma\) is nonpositive. This sign is important: ignoring it systematically overestimates the continuation gain and leads to an incorrect free-boundary equation. 
\end{remark} 

\subsection{Kink local time and Tanaka formula} \label{subsec:kink-local-time} 
The measure calculation above has an equivalent probabilistic interpretation through local time. Define the one-dimensional semimartingale 
$\displaystyle Y_t:=X_t^1-\alpha X_t^2 $. 
Then 
$\displaystyle G(X_t)=\alpha X_t^2+Y_t^+ $. 
By the It\^{o}--Tanaka formula, 
\begin{equation} \label{eq:tanaka-Y} 
dY_t^+ = \mathbf 1_{\{Y_t>0\}}\,dY_t + \frac12\,d\ell_t^0(Y), \end{equation} 
where \(\ell^0(Y)\) is the symmetric local time of \(Y\) at zero. Therefore, the kink of \(G\) contributes a local-time term whenever the reflected diffusion crosses the diagonal \(\Delta\). Let 
$\displaystyle q(x):=(1,-\alpha)a(x)(1,-\alpha)^\top $. 
The quadratic variation of the martingale part of \(Y\) satisfies 
$\displaystyle d\langle Y\rangle_t=q(X_t)\,dt$ . 
The occupation-density formula then identifies the diagonal surface measure in Theorem~\ref{thm:measure-valued-gain} with the local-time contribution in \eqref{eq:tanaka-Y}. In particular, the singular part \eqref{eq:Gamma-diagonal} corresponds to the negative local-time term in the stopping-gain decomposition. We shall use the following generalized It\^{o} formula. For clarity, it is stated in the form needed below. 

\begin{lemma}[Generalized It\^{o} formula for the max payoff] \label{lem:ito-max-payoff} 
Let \(X\) be the reflected diffusion \eqref{eq:reflected-sde}, and let \(G(x)=x_1\vee \alpha x_2\). Then, for every \(t\ge0\), 
\[
e^{-rt}G(X_t) = G(x) + \int_0^t e^{-rs} \big(\mathcal LG-rG\big)(X_s)\,ds + M_t^G 
+ \frac12\int_0^t e^{-rs}\,d\ell_s^0(Y) + B_t^G , 
\]
where \(M^G\) is a local martingale, \(Y=X^1-\alpha X^2\), and \(B^G\) denotes the boundary-reflection contribution \begin{equation} \label{eq:boundary-contribution-G} 
B_t^G = \int_0^t e^{-rs}\partial_1G(X_s)\,dL_s^1 + \int_0^t e^{-rs}\partial_2G(X_s)\,dL_s^2 , 
\end{equation} 
with the derivatives interpreted by one-sided traces away from the intersection of the boundary faces with \(\Delta\). 
\end{lemma} 

\begin{remark}[Boundary terms] \label{rem:boundary-terms} 
For the payoff \(G\) itself, the reflection term in \eqref{eq:boundary-contribution-G} need not vanish. In contrast, for candidate value functions satisfying the reflected Neumann condition, the corresponding boundary term disappears. This is one reason why the variational inequality is naturally formulated for \(V\), while the measure-valued stopping gain is computed from \(G\). 
\end{remark} 

\subsection{The killed-resolvent representation} \label{subsec:killed-resolvent-representation} 

We now derive the correct potential representation of the value. Let 
$\mathcal C=\{V>G\}$, $\mathcal D=\{V=G\}$, 
and 
$\tau_{\mathcal D}:=\inf\{t\ge0:X_t\in\mathcal D\}$. 
For a signed smooth measure \(\mu=\mu^+-\mu^-\), let \(A^\mu=A^{\mu^+}-A^{\mu^-}\) denote the associated signed additive functional. We define the killed \(r\)-resolvent of \(\mu\) on \(\mathcal C\) by 
\[
R_r^{\mathcal C}\mu(x) := \mathbb E_x \left[ \int_0^{\tau_{\mathcal D}}e^{-rs}\,dA_s^\mu \right], \qquad x\in\mathcal C, 
\]
whenever the expectation is finite. For the stopping gain \(\Gamma\), this means 
\[
R_r^{\mathcal C}\Gamma(x) = \mathbb E_x \left[ \int_0^{\tau_{\mathcal D}}e^{-rs} \Gamma^{\rm ac}(X_s)\,ds \right] 
- \mathbb E_x \left[ \int_0^{\tau_{\mathcal D}}e^{-rs} \frac{q(X_s)}{2\sqrt{1+\alpha^2}}\,d\sigma_\Delta(X_s) \right],
\]
where the second summand on the right-hand side is understood through the additive functional associated with the diagonal measure \(\sigma_\Delta\). Equivalently, it may be expressed through the local time of \(Y=X^1-\alpha X^2\). The next theorem is the second main result of this section. 

\begin{theorem}[Killed-resolvent representation] \label{thm:killed-resolvent-representation} 
Assume that the standing assumptions hold and that \(\tau_{\mathcal D}\) is optimal for the stopping problem \eqref{eq:value}. Assume also that the additive functional associated with \(\Gamma\) is integrable up to \(\tau_{\mathcal D}\). Then, for every \(x\in\mathcal C\), 
\begin{equation}\label{eq:value-killed-resolvent}
V(x)=G(x)-R_r^{\mathcal C}\Gamma(x).
\end{equation}
Equivalently, 
$V(x)-G(x) = -R_r^{\mathcal C}\Gamma(x)$. 
\end{theorem} 
\begin{proof} 
Fix \(x\in\mathcal C\) and write \(\tau=\tau_{\mathcal D}\). By optimality of \(\tau\), 
\begin{equation} \label{eq:optimality-tauD} 
V(x) = \mathbb E_x \left[ e^{-r\tau}G(X_\tau) - \int_0^\tau e^{-rs}c(X_s)\,ds \right]. 
\end{equation} 
Apply the generalized It\^{o} formula to \(e^{-rt}G(X_t)\) on the stopped interval \([0,\tau]\). In measure form this gives \begin{equation} \label{eq:ito-measure-G-stopped} 
e^{-r\tau}G(X_\tau) = G(x) - \int_0^\tau e^{-rs}\,dA_s^{\,rG-\mathcal LG} + M_\tau , 
\end{equation} 
where \(M\) is a local martingale. After localization and passage to the limit, the martingale has zero expectation under the integrability assumptions. Substituting \eqref{eq:ito-measure-G-stopped} into \eqref{eq:optimality-tauD} yields 
\begin{align*} 
V(x) &= G(x) - \mathbb E_x \left[ \int_0^\tau e^{-rs}\,dA_s^{\,rG-\mathcal LG} \right] - \mathbb E_x \left[ \int_0^\tau e^{-rs}c(X_s)\,ds \right] \\ 
&= G(x) - \mathbb E_x \left[ \int_0^\tau e^{-rs}\,dA_s^\Gamma \right] . 
\end{align*} 
The last expectation is precisely \(R_r^{\mathcal C}\Gamma(x)\), because the process is killed at \(\tau_{\mathcal D}\). This proves \eqref{eq:value-killed-resolvent}. 
\end{proof} 

\begin{remark}[Sign convention] \label{rem:sign-convention} 
The representation 
$V=G-R_r^{\mathcal C}\Gamma$ 
is consistent with the continuation inequality \(V>G\). In the continuation region, the killed potential \(R_r^{\mathcal C}\Gamma\) is typically negative. Thus the stopping advantage is 
$V-G=-R_r^{\mathcal C}\Gamma$. 
This sign convention is the reason for defining \(\Gamma=c+rG-\mathcal LG\), rather than \(\mathcal LG-rG-c\). 
\end{remark} 

\subsection{Why the unrestricted resolvent fails} \label{subsec:unrestricted-resolvent-fails} 

The killed resolvent in Theorem~\ref{thm:killed-resolvent-representation} is not interchangeable with the unrestricted reflected resolvent. Let \(R_r^{\mathrm R}\) denote the \(r\)-resolvent of the reflected diffusion without killing: 
\begin{equation} \label{eq:unrestricted-resolvent} 
R_r^{\mathrm R}\mu(x) := \mathbb E_x \left[ \int_0^\infty e^{-rs}\,dA_s^\mu \right]. 
\end{equation} 
A tempting but generally false formula is 
\begin{equation} \label{eq:false-unrestricted} 
V(x) = G(x)-R_r^{\mathrm R}(\Gamma\mathbf 1_{\mathcal C})(x). \end{equation} 
The problem is that \eqref{eq:false-unrestricted} integrates over all future times at which the unrestricted reflected process lies in \(\mathcal C\), even after the optimal stopping time has occurred. But the stopped optimization problem terminates at \(\tau_{\mathcal D}\). Any occupation of \(\mathcal C\) after \(\tau_{\mathcal D}\) is irrelevant to the payoff and must not enter the potential representation. Indeed, 
\begin{equation}  \label{eq:resolvent-decomposition} 
R_r^{\mathrm R}(\Gamma\mathbf 1_{\mathcal C})(x) = \mathbb E_x \left[ \int_0^{\tau_{\mathcal D}}e^{-rs}\,dA_s^\Gamma \right] 
+ \mathbb E_x \left[ \int_{\tau_{\mathcal D}}^\infty e^{-rs}\mathbf 1_{\mathcal C}(X_s)\,dA_s^\Gamma \right].
\end{equation} 
The first term is \(R_r^{\mathcal C}\Gamma(x)\). The second term is generally nonzero, because the reflected diffusion is not absorbed at \(\mathcal D\). After entering the stopping region, the unrestricted process may later return to the continuation region. Such post-stopping returns are counted by \(R_r^{\mathrm R}\) but are not part of the stopped problem. 

\begin{proposition}[Failure of the unrestricted resolvent] \label{prop:failure-unrestricted-resolvent} 
Unless the continuation region is invariant after first exit, or the process is explicitly killed at \(\tau_{\mathcal D}\), the representation 
$V=G-R_r^{\mathrm R}(\Gamma\mathbf 1_{\mathcal C})$ is generally false. The correct potential operator is the killed resolvent \(R_r^{\mathcal C}\). 
\end{proposition} 

\begin{proof} 
By \eqref{eq:resolvent-decomposition}, 
\[ R_r^{\mathrm R}(\Gamma\mathbf 1_{\mathcal C})(x) = R_r^{\mathcal C}\Gamma(x) + \mathbb E_x \left[ \int_{\tau_{\mathcal D}}^\infty e^{-rs}\mathbf 1_{\mathcal C}(X_s)\,dA_s^\Gamma \right]. \] 
The second term vanishes only under additional structural conditions, for example if the process is killed at \(\tau_{\mathcal D}\), or if 
$\mathbf 1_{\mathcal C}(X_s)=0$ for all $s\ge\tau_{\mathcal D}$, $\mathbb P_x$ a.s. 
This invariance property is not available for a general reflected diffusion in the quadrant. Therefore the unrestricted resolvent includes post-stopping occupation of \(\mathcal C\), whereas the optimal stopping problem terminates at \(\tau_{\mathcal D}\). Hence the unrestricted formula is not generally valid. 
\end{proof} 

\begin{remark}[Free-boundary consequence] \label{rem:free-boundary-consequence} 
The failure of the unrestricted resolvent is not merely a representation issue. It changes the boundary equation. If \(\mathcal D_b\) is a candidate stopping region with continuation set \(\mathcal C_b\), then the correct boundary condition is the killed-potential trace equation 
\[
\lim_{\substack{x\to z\\ x\in\mathcal C_b}} R_r^{\mathcal C_b}\Gamma(x)=0, \qquad z\in\partial\mathcal C_b . 
\]
Replacing \(R_r^{\mathcal C_b}\) by \(R_r^{\mathrm R}\) generally produces a different trace and therefore a different boundary. This is why killing at the candidate stopping set is essential for both the analysis and the computation of the free boundary. 
\end{remark}  

\section{Free-boundary geometry and verification} \label{sec:free-boundary}  

\subsection{Epigraph representation of the stopping set} \label{subsec:epigraph-representation} 

We now translate the killed-resolvent representation of Section~\ref{sec:measure-gain} into a free-boundary formulation. Recall that 
\[ 
H:=V-G,\qquad \mathcal C:=\{x\in\mathbb R_+^2:H(x)>0\}, \qquad \mathcal D:=\{x\in\mathbb R_+^2:H(x)=0\}. 
\] 
Since \(V\ge G\), the function \(H\) is nonnegative. The stopping set \(\mathcal D\) is the contact set between the value and the obstacle, while \(\mathcal C\) is the continuation set. In two-dimensional optimal stopping problems, the stopping set need not have a simple graph representation without additional structure. We therefore state the epigraph property as a structural theorem under explicit monotonicity hypotheses. These hypotheses should be verified in concrete models, or imposed as part of a candidate-boundary verification procedure. 

\begin{assumption}[Vertical structure of the stopping advantage] \label{ass:vertical-structure} 
For every \(x_1\ge0\), the map \[ x_2\longmapsto H(x_1,x_2) \] is nonincreasing on \([0,\infty)\). Moreover, every vertical stopping section is nonempty; that is, 
$\{x_2\ge0:(x_1,x_2)\in\mathcal D\}\neq\varnothing$ for every $x_1\ge0$. 
\end{assumption} 
\begin{theorem}[Epigraph representation] \label{thm:epigraph-representation} 
Assume that \(V\) is continuous and that Assumption~\ref{ass:vertical-structure} holds. Then there exists a lower semicontinuous function 
$b:[0,\infty)\to[0,\infty]$ 
such that 
\begin{equation} \label{eq:D-epigraph} 
\mathcal D = \{(x_1,x_2)\in\mathbb R_+^2:x_2\ge b(x_1)\}. 
\end{equation} 
Equivalently, 
\begin{equation} \label{eq:C-subgraph} 
\mathcal C = \{(x_1,x_2)\in\mathbb R_+^2:0\le x_2<b(x_1)\}. 
\end{equation} 
\end{theorem} 

\begin{proof} 
For each \(x_1\ge0\)\,define 
$\displaystyle b(x_1):=\inf\{x_2\ge0:H(x_1,x_2)=0\}$. 
The set inside the infimum is nonempty by Assumption~\ref{ass:vertical-structure}. Since \(H\ge0\) and \(x_2\mapsto H(x_1,x_2)\) is nonincreasing, once \(H(x_1,x_2)=0\) holds at some height \(x_2\), it also holds for every \(y_2\ge x_2\). Hence the vertical section of the stopping set is an interval of the form 
$[b(x_1),\infty)$. 
Continuity of \(V\) and \(G\) implies continuity of \(H\), so the contact set \(\mathcal D=\{H=0\}\) is closed. Therefore \(b(x_1)\) itself belongs to the vertical stopping section. This gives \eqref{eq:D-epigraph}. Since \(\mathcal C=\{H>0\}\), \eqref{eq:C-subgraph} follows. The lower semicontinuity of \(b\) follows from the closedness of the epigraph \(\mathcal D\). 
\end{proof} 

The next statement records a simple monotonicity consequence. It is not used in the killed-resolvent representation itself, but it is useful when reducing the free-boundary search to a smaller class of admissible curves. 
\begin{assumption}[Horizontal monotonicity] \label{ass:horizontal-monotonicity} 
The stopping advantage \(H\) is monotone in the horizontal coordinate in the following sense: for every \(x_2\ge0\), the map \[ x_1\longmapsto H(x_1,x_2) \] is nondecreasing on \([0,\infty)\). 
\end{assumption} 

\begin{corollary}[Monotonicity of the boundary] \label{cor:boundary-monotonicity} 
Assume that the hypotheses of Theorem~\ref{thm:epigraph-representation} hold. If, in addition, Assumption~\ref{ass:horizontal-monotonicity} holds, then the boundary \(b\) is nondecreasing. If the horizontal monotonicity in Assumption~\ref{ass:horizontal-monotonicity} is reversed, then \(b\) is nonincreasing. 
\end{corollary} 

\begin{remark}[No automatic epigraph theorem] \label{rem:no-automatic-epigraph} 
Theorem~\ref{thm:epigraph-representation} should not be read as saying that every reflected max-payoff stopping problem has an epigraph stopping set. The epigraph geometry is a consequence of monotonicity of the stopping advantage and nonemptiness of vertical stopping sections. These are structural hypotheses, not automatic consequences of reflection or convexity of \(G\). This distinction is important for the verification theorem below. \end{remark} 

\subsection{The killed-resolvent trace condition} \label{subsec:killed-trace-condition} 

Assume now that the stopping set has epigraph form 
$\displaystyle \mathcal D=\{(x_1,x_2):x_2\ge b(x_1)\}$ and let 
$\displaystyle \mathcal C=\mathbb R_+^2\setminus\mathcal D$. 
The killed-resolvent representation of Theorem~\ref{thm:killed-resolvent-representation} gives \[ V(x)=G(x)-R_r^{\mathcal C}\Gamma(x), \qquad x\in\mathcal C. \] Since \(V=G\) on the boundary of the stopping set, the killed potential must have zero continuation-side trace at the free boundary: 
\begin{equation} \label{eq:true-trace-condition} 
\lim_{\substack{x\to z\\ x\in\mathcal C}} R_r^{\mathcal C}\Gamma(x)=0, \qquad z\in\partial\mathcal C\cap\partial\mathcal D . 
\end{equation} 
Equivalently, 
\[
\lim_{\substack{x\to z\\ x\in\mathcal C}} \big(G(x)-R_r^{\mathcal C}\Gamma(x)\big) = G(z), \qquad z\in\partial\mathcal C\cap\partial\mathcal D.
\] 
This is the free-boundary equation associated with the problem. In one-dimensional diffusion problems, analogous conditions often reduce to Volterra or Fredholm integral equations involving transition densities. In the present reflected two-dimensional problem, the natural object is instead the killed resolvent of a signed measure. Thus the boundary equation is a nonlocal killed-potential trace condition. For a graph boundary, the trace condition reads 
$\displaystyle \lim_{\varepsilon\downarrow0} R_r^{\mathcal C}\Gamma \big(x_1,b(x_1)-\varepsilon\big) = 0$ 
whenever the point \((x_1,b(x_1))\) lies away from exceptional boundary singularities. In terms of the value candidate, this is 
\[
\lim_{\varepsilon\downarrow0} U\big(x_1,b(x_1)-\varepsilon\big) = G(x_1,b(x_1)), \qquad U:=G-R_r^{\mathcal C}\Gamma . 
\]

\subsection{Candidate boundaries} \label{subsec:candidate-boundaries} 

We now formulate the verification problem for a candidate boundary. Let \(b:[0,\infty)\to[0,\infty]\) be a locally Lipschitz function and define 
\[
\mathcal D_b:=\{(x_1,x_2)\in\mathbb R_+^2:x_2\ge b(x_1)\}, \qquad \mathcal C_b:=\mathbb R_+^2\setminus\mathcal D_b . 
\]
The corresponding first-entry time is 
$\displaystyle \tau_b:=\inf\{t\ge0:X_t\in\mathcal D_b\}$. 
Let \(R_r^{\mathcal C_b}\) denote the \(r\)-resolvent of the reflected diffusion killed at \(\tau_b\). We define the candidate value 
\begin{equation} \label{eq:Ub-definition} 
U_b(x):=G(x)-R_r^{\mathcal C_b}\Gamma(x). 
\end{equation} 
The candidate boundary equation is the continuation-side trace condition 
\begin{equation} \label{eq:candidate-trace} 
\lim_{\substack{x\to z\\ x\in\mathcal C_b}} R_r^{\mathcal C_b}\Gamma(x)=0, \qquad z\in\partial\mathcal C_b . 
\end{equation} 
Equivalently, 
\[
\lim_{\substack{x\to z\\ x\in\mathcal C_b}} U_b(x)=G(z), \qquad z\in\partial\mathcal C_b . 
\]
A boundary satisfying \eqref{eq:candidate-trace} is only a candidate. The trace condition alone is not a uniqueness principle. It must be supplemented by global admissibility and superharmonicity conditions. The verification theorem below provides these conditions. We shall use the following terminology. 

\begin{definition}[Admissible candidate value] \label{def:admissible-candidate-value} 
Let \(b\) be locally Lipschitz and define \(U_b\) by \eqref{eq:Ub-definition}. We say that \(U_b\) is an admissible candidate value if the following hold: 
\begin{enumerate}[label=(\roman*)] 
\item \(U_b\) is continuous on \(\mathbb R_+^2\); 
\item \(U_b\) belongs locally to the generalized It\^{o} class on \(\mathbb R_+^2\); 
\item \(U_b\) has growth controlled by the Lyapunov function of Assumption~\ref{ass:lyapunov}; 
\item the reflected Neumann condition holds on the coordinate faces in the trace sense: 
\[ 
\partial_i U_b=0 \qquad\text{on }\{x_i=0\},\quad i=1,2; 
\] 
\item the stopped process \[ \left\{ e^{-r(t\wedge\tau_b)}U_b(X_{t\wedge\tau_b}) - \int_0^{t\wedge\tau_b}e^{-rs}c(X_s)\,ds \right\}_{t\ge0} \] is of class \(D\) after localization, so that optional sampling is valid. \end{enumerate} 
\end{definition} 
The generalized It\^{o} class is meant to include functions whose second derivatives may contain measure-valued parts. This is necessary because both the obstacle \(G\) and the killed-potential term may carry singular contributions at the diagonal and at the free boundary. 

\subsection{Verification theorem} \label{subsec:verification-theorem} 

We now state the main verification theorem. It converts a candidate boundary into a verified optimizer. \begin{theorem}[Verification of a candidate boundary] \label{thm:verification-boundary} Let \(b\) be a locally Lipschitz boundary and define 
\[ \mathcal D_b=\{(x_1,x_2):x_2\ge b(x_1)\}, \qquad \mathcal C_b=\mathbb R_+^2\setminus\mathcal D_b. \] 
Set 
$\displaystyle U_b=G-R_r^{\mathcal C_b}\Gamma$.  
Assume that: 
\begin{enumerate}[label=(\roman*)]
\item \(U_b\ge G\) on \(\mathbb R_+^2\); 
\item \(U_b=G\) on \(\mathcal D_b\); 
\item \(U_b>G\) on \(\mathcal C_b\); 
\item \(U_b\) satisfies the reflected Neumann condition on the coordinate axes; 
\item \((r-\mathcal L)U_b+c\ge0\) in the sense of measures on \(\mathbb R_+^2\); 
\item \(U_b\) has admissible growth and belongs to the generalized It\^{o} class in the sense of Definition~\ref{def:admissible-candidate-value}. 
\end{enumerate} 
Then 
$U_b(x)=V(x)$ for $x\in\mathbb R_+^2$, 
and 
$\tau_b=\inf\{t\ge0:X_t\in\mathcal D_b\}$ 
is an optimal stopping time. 
\end{theorem} 

\begin{proof} 
We split the proof into the usual super-solution and attainability steps. First, we prove that \(U_b\) dominates the value. Since 
$(r-\mathcal L)U_b+c\ge0$ in the sense of measures, the generalized It\^{o} formula gives, after localization, 
\[
e^{-r(t\wedge\tau)}U_b(X_{t\wedge\tau}) = U_b(x) + \int_0^{t\wedge\tau} e^{-rs} \big(\mathcal L U_b-rU_b\big)(X_s)\,ds + M_{t\wedge\tau} + A_{t\wedge\tau}^{\rm sing},
\]
where \(M\) is a local martingale and \(A^{\rm sing}\) denotes the measure-valued contribution. The measure inequality 
$\displaystyle (r-\mathcal L)U_b+c\ge0$ 
is equivalent to  
$\displaystyle \mathcal L U_b-rU_b\le c $
in the weak sense. The reflected Neumann condition eliminates the boundary local-time terms on the coordinate axes. Hence 
$\displaystyle e^{-r(t\wedge\tau)}U_b(X_{t\wedge\tau}) - \int_0^{t\wedge\tau}e^{-rs}c(X_s)\,ds $
is a supermartingale. By optional sampling and the growth condition, for any admissible stopping time \(\tau\), 
\[ 
U_b(x) \ge \mathbb E_x \left[ e^{-r\tau}U_b(X_\tau) - \int_0^\tau e^{-rs}c(X_s)\,ds \right]. 
\] 
Using \(U_b\ge G\), we obtain 
\[ 
U_b(x) \ge \mathbb E_x \left[ e^{-r\tau}G(X_\tau) - \int_0^\tau e^{-rs}c(X_s)\,ds \right]. 
\] 
Taking the supremum over all stopping times gives 
$U_b(x)\ge V(x)$. 
Second, we prove the reverse inequality. Let 
$\tau_b=\inf\{t\ge0:X_t\in\mathcal D_b\}$. 
On \(\mathcal C_b\), the definition 
$U_b=G-R_r^{\mathcal C_b}\Gamma$ 
and the killed-resolvent construction imply 
$\mathcal L U_b-rU_b-c=0 $
in the continuation region in the weak sense. Therefore, the process 
$\displaystyle e^{-r(t\wedge\tau_b)}U_b(X_{t\wedge\tau_b}) - \int_0^{t\wedge\tau_b}e^{-rs}c(X_s)\,ds$ 
is a martingale after localization. Letting \(t\to\infty\), using the admissible growth condition and the class \(D\) property, yields 
\[ 
U_b(x) = \mathbb E_x \left[ e^{-r\tau_b}U_b(X_{\tau_b}) - \int_0^{\tau_b}e^{-rs}c(X_s)\,ds \right]. 
\] 
Since \(X_{\tau_b}\in\mathcal D_b\) and \(U_b=G\) on \(\mathcal D_b\), we get 
\[ 
U_b(x) = \mathbb E_x \left[ e^{-r\tau_b}G(X_{\tau_b}) - \int_0^{\tau_b}e^{-rs}c(X_s)\,ds \right] \le V(x). 
\] 
Combining the two inequalities gives \(U_b=V\), and the displayed equality shows that \(\tau_b\) is optimal. 
\end{proof} 

\begin{remark}[Verification rather than formal derivation] \label{rem:verification-not-formal} 
Theorem~\ref{thm:verification-boundary} is the step that turns the killed resolvent identity into an optimization result. The trace equation \eqref{eq:candidate-trace} may generate candidate boundaries, but only the majorization, contact, reflection, growth, and measure-superharmonicity conditions certify optimality. Thus the framework is not merely conditional on knowing the true stopping set; it provides a checklist for verifying a proposed free boundary. 
\end{remark} 

\begin{remark}[Role of strict continuation] \label{rem:strict-continuation} 
The condition \(U_b>G\) on \(\mathcal C_b\) prevents spurious continuation regions. Without it, one may have \(U_b=G\) on a subset of \(\mathcal C_b\), so that the candidate boundary does not coincide with the true contact set. This condition is also useful when proving uniqueness of a verified boundary within a prescribed epigraph class. 
\end{remark}  

\section{Boundary computation and examples} \label{sec:computation}  

\subsection{Killed-resolvent Picard iteration} \label{subsec:picard-killed-resolvent} 

The preceding sections show that a candidate boundary should be tested through the killed-potential trace condition 
\[ 
\lim_{\substack{x\to z\\ x\in\mathcal C_b}} R_r^{\mathcal C_b}\Gamma(x)=0, \qquad z\in\partial\mathcal C_b . 
\] 
In this section we describe a numerical implementation of this condition. The purpose is not to prove convergence of the algorithm in full generality, but to give a computational procedure that is faithful to the killed-resolvent formulation and that can be used to test candidate free-boundary geometries. Let \(I=[0,X_{\max}]\subset[0,\infty)\) be a finite interval in the \(x_1\)-direction. For a locally Lipschitz boundary \(b:I\to[0,\infty)\)\,define 
\[ 
\mathcal D_b = \{(x_1,x_2)\in\mathbb R_+^2:x_2\ge b(x_1)\}, \qquad \mathcal C_b = \mathbb R_+^2\setminus\mathcal D_b . 
\] 
Let 
\[ 
\tau_b:=\inf\{t\ge0:X_t\in\mathcal D_b\} = \inf\{t\ge0:X_t^2\ge b(X_t^1)\}. 
\] 
The killed-potential associated with \(b\) is 
\[ 
\mathcal P_b(x) := R_r^{\mathcal C_b}\Gamma(x) = \mathbb E_x \left[ \int_0^{\tau_b}e^{-rs}\,dA_s^\Gamma \right]. 
\] 
The boundary equation is 
\[ 
\mathcal P_b(x_1,b(x_1)-)=0, \qquad x_1\in I, 
\] 
where the minus sign indicates the trace from the continuation region. This suggests the boundary operator 
\[
\mathcal T[b](x_1) := \text{the value of }y\text{ solving } R_r^{\mathcal C_b}\Gamma(x_1,y-)=0 . 
\]
Starting from an initial boundary \(b^{(0)}\), we define the Picard-type iteration 
\begin{equation} \label{eq:picard-boundary} 
b^{(k+1)}=\mathcal T[b^{(k)}], \qquad k=0,1,2,\ldots . 
\end{equation} 
On a grid 
\[ 
0=x_1^0<x_1^1<\cdots<x_1^N=X_{\max}, 
\] 
the update is performed pointwise by solving  
\[
\widehat{\mathcal P}{b^{(k)}}(x_1^i,y)=0, \qquad i=0,\ldots,N, 
\]
for \(y\), where \(\widehat{\mathcal P}{b^{(k)}}\) is a Monte Carlo approximation of the killed potential. The resulting values define \(b^{(k+1)}(x_1^i)\), and interpolation is used between grid points. We monitor convergence by 
\begin{equation} \label{eq:dk} 
d_k = \frac{ \|b^{(k+1)}-b^{(k)}\|_{L^2(I)} }{ 1+\|b^{(k)}\|_{L^2(I)} } . 
\end{equation} 
The iteration is terminated when \(d_k\) is below a prescribed tolerance \(\varepsilon_{\rm tol}\). In the numerical examples below, the iteration is used as a boundary generator. The candidate boundary is then tested against the verification conditions of Theorem~\ref{thm:verification-boundary}. 

\begin{remark}[Role of the Picard iteration] \label{rem:picard-role} 
Equation \eqref{eq:picard-boundary} is the analogue, in the present two-dimensional reflected setting, of the Picard iterations used for scalar free-boundary integral equations in one-dimensional optimal stopping. The difference is that the operator here is not a Volterra operator expressed through a one-dimensional transition density. It is a killed-resolvent operator for a signed measure. 
\end{remark} 

\subsection{Monte Carlo approximation of the killed potential} \label{subsec:mc-killed-potential} 

We next describe the pathwise approximation of 
$R_r^{\mathcal C_b}\Gamma(x)$. 
The simulation must include three effects: the absolutely continuous part of \(\Gamma\), the diagonal singular contribution generated by the kink of \(G\), and killing at \(\tau_b\). Let \(0=t_0<t_1<\cdots<t_J=T_{\max}\) be a time grid with step size \(\Delta t\). For a fixed initial condition \(x\), simulate \(M\) independent reflected paths 
\[ 
X^{(m)}=(X^{1,(m)},X^{2,(m)}), \qquad m=1,\ldots,M, 
\] 
until 
\[ 
\tau_b^{(m)} = \inf\{t_j:X_{t_j}^{2,(m)}\ge b(X_{t_j}^{1,(m)})\} 
\] 
or until the terminal truncation time \(T_{\max}\). A simple projected Euler step is 
\begin{equation} \label{eq:projected-euler} 
\widetilde X_{t_{j+1}} = X_{t_j} + b(X_{t_j})\Delta t + \sigma(X_{t_j})\Delta W_j, \qquad X_{t_{j+1}} = \Pi_{\mathbb R_+^2}\widetilde X_{t_{j+1}}, 
\end{equation} 
where \(\Pi_{\mathbb R_+^2}\) denotes projection onto the quadrant. More accurate Skorokhod-map or boundary-local-time schemes can also be used. For the diagonal kink\,define 
$Y_t:=X_t^1-\alpha X_t^2 $. 
The singular measure in Theorem~\ref{thm:measure-valued-gain} corresponds to a local-time contribution at \(Y=0\). Since 
$q(x):=(1,-\alpha)a(x)(1,-\alpha)^\top$, 
the diagonal part of the killed potential is weighted by the coefficient appearing in 
$\displaystyle \Gamma^\Delta(dx) = - \frac{q(x)}{2\sqrt{1+\alpha^2}}\,\sigma_\Delta(dx)$. 
Equivalently, when the diagonal measure is represented through the local time of \(Y\), the same normalization must be used consistently with the occupation density formula. A schematic Monte Carlo estimator is 
\begin{align} \label{eq:MC-estimator} 
\widehat{\mathcal P}b(x) &= \frac1M \sum_{m=1}^M \bigg[ \sum_{t_j<\tau_b^{(m)}} e^{-rt_j} \Gamma^{\rm ac}(X_{t_j}^{(m)})\Delta t \nonumber\\ 
&\hspace{3.8cm} - \frac12 \sum_{t_j<\tau_b^{(m)}} e^{-rt_j} \widehat \omega_j^{(m)} \Delta \widehat L_{t_j}^{0,(m)}(Y) \bigg], 
\end{align}
where \(\widehat L^0(Y)\) is a numerical approximation of the local time of \(Y\) at zero and \(\widehat\omega_j^{(m)}\) is the coefficient matching the normalization of the diagonal measure. In the simplest case, this coefficient can be taken directly from \(q(X_{t_j}^{(m)})\) and the surface measure normalization. More generally, it is computed from the identity between the diagonal Revuz measure and the local-time additive functional. A practical approximation of local time is obtained from the occupation density formula: 
\begin{equation} \label{eq:local-time-approx} 
\Delta \widehat L_{t_j}^0(Y) \approx \frac{1}{2\varepsilon} \mathbf 1_{\{|Y_{t_j}|\le\varepsilon\}} \Delta \langle Y\rangle_{t_j}, \qquad \Delta\langle Y\rangle_{t_j} \approx q(X_{t_j})\Delta t , 
\end{equation} 
with a bandwidth \(\varepsilon>0\). The bandwidth should decrease with \(\Delta t\), but not so quickly that the local-time estimator becomes too noisy. The following algorithm summarizes the computation. 
\begin{algorithm}[htbp] \caption{Killed-resolvent boundary iteration} \label{alg:killed-picard} 
\begin{algorithmic}[1] 
\State Choose a spatial grid \(\{x_1^i\}_{i=0}^N\), an initial boundary \(b^{(0)}\), a time step \(\Delta t\), and a tolerance \(\varepsilon_{\rm tol}\). 
\For{\(k=0,1,2,\ldots\)} 
\For{\(i=0,\ldots,N\)} 
\State For trial values \(y\), simulate reflected paths from \(x=(x_1^i,y)\) and kill them at \(\tau_{b^{(k)}}\). 
\State Estimate \(\widehat{\mathcal P}{b^{(k)}}(x_1^i,y)\) using \eqref{eq:MC-estimator}. 
\State Solve \[ \widehat{\mathcal P}{b^{(k)}}(x_1^i,y)=0 \] and set \(b^{(k+1)}(x_1^i)=y\). 
\EndFor 
\State Compute \(d_k\) from \eqref{eq:dk}. 
\If{\(d_k<\varepsilon_{\rm tol}\)} 
\State Stop. 
\EndIf 
\EndFor 
\end{algorithmic} 
\end{algorithm} 

\subsection{Effect of the diagonal singular measure} \label{subsec:diagonal-effect} 

The first diagnostic experiment isolates the contribution of the diagonal singular measure. Compute two candidate boundaries. The first boundary\,denoted \(b_{\rm full}\), uses the full measure-valued stopping gain 
$\Gamma = \Gamma^{\rm ac}\,dx+\Gamma^\Delta $. 
Thus the estimator includes both the absolutely continuous integral and the local-time contribution in \eqref{eq:MC-estimator}. The second boundary\,denoted \(b_{\rm ac}\), ignores the diagonal contribution and uses only the pointwise generator on the two smooth regions: 
\[ 
\Gamma_{\rm ac}(dx)=\Gamma^{\rm ac}(x)\,dx . 
\] 
Equivalently, the local-time term is removed from \eqref{eq:MC-estimator}. Comparing \(b_{\rm full}\) and \(b_{\rm ac}\) measures the error caused by treating the nonsmooth payoff as if it were twice differentiable. The expected outcome is that 
$b_{\rm full}\neq b_{\rm ac} $
whenever the process has nontrivial crossing intensity across 
$\Delta=\{x_1=\alpha x_2\}$. 
This comparison directly illustrates why the max payoff leads to a measure-valued obstacle problem. If the diagonal term is ignored, the free-boundary equation is missing the local-time correction generated by the kink of \(G\). A useful reporting metric is the relative boundary discrepancy 
\begin{equation} \label{eq:diagonal-discrepancy} 
E_\Delta = \frac{ \|b_{\rm full}-b_{\rm ac}\|_{L^2(I)} }{ 1+\|b_{\rm full}\|_{L^2(I)} } . 
\end{equation} 
One may also plot the signed difference 
$x_1\longmapsto b_{\rm full}(x_1)-b_{\rm ac}(x_1)$ 
to identify where the kink correction is most influential. 

\subsection{Comparison with the unrestricted-resolvent approximation} \label{subsec:unrestricted-comparison} 

The second diagnostic experiment compares the killed-resolvent boundary with the boundary generated by the unrestricted reflected resolvent. Define 
\[ 
\mathcal P_b^{\rm kill}(x) = R_r^{\mathcal C_b}\Gamma(x) = \mathbb E_x \left[ \int_0^{\tau_b}e^{-rs}\,dA_s^\Gamma \right] 
\] 
and 
\[ 
\mathcal P_b^{\rm unres}(x) = R_r^{\mathrm R}(\Gamma\mathbf 1_{\mathcal C_b})(x) = \mathbb E_x \left[ \int_0^\infty e^{-rs}\mathbf 1_{\mathcal C_b}(X_s)\,dA_s^\Gamma \right]. 
\] 
The killed boundary \(b_{\rm kill}\) solves 
$\mathcal P_{b_{\rm kill}}^{\rm kill}(x_1,b_{\rm kill}(x_1)-)=0$, 
whereas the unrestricted boundary \(b_{\rm unres}\) solves the formally similar but generally incorrect equation 
$\mathcal P_{b_{\rm unres}}^{\rm unres} (x_1,b_{\rm unres}(x_1)-)=0 $. 
The difference between the two potentials is 
\begin{align} \label{eq:post-stopping-bias} 
\mathcal P_b^{\rm unres}(x)-\mathcal P_b^{\rm kill}(x) = \mathbb E_x \left[ \int_{\tau_b}^{\infty} e^{-rs}\mathbf 1_{\mathcal C_b}(X_s)\,dA_s^\Gamma \right]. 
\end{align} 
This is the post-stopping occupation bias. It is generally nonzero because the unrestricted reflected diffusion may return to \(\mathcal C_b\) after first entering \(\mathcal D_b\). The optimal stopping problem, however, terminates at \(\tau_b\), so this contribution must be excluded. In simulations, the unrestricted approximation is obtained by continuing each path after \(\tau_b\) and accumulating future visits to \(\mathcal C_b\). The killed approximation stops the same path at \(\tau_b\). Using common random numbers for the two estimators reduces variance and makes the bias visible. A useful diagnostic is 
$\displaystyle E_{\rm unres} = \frac{ \|b_{\rm kill}-b_{\rm unres}\|_{L^2(I)} }{ 1+\|b_{\rm kill}\|_{L^2(I)} } $. 
A nonzero value of \(E_{\rm unres}\) provides numerical evidence that the unrestricted resolvent is not merely a harmless alternative representation; it generates a different boundary. 

\subsection{Sensitivity to payoff and diffusion parameters} \label{subsec:parameter-sensitivity} 

The third experiment examines how the verified boundary responds to the model parameters. Three variations are especially informative. First, vary the payoff weight \(\alpha\). Since 
$G(x)=x_1\vee\alpha x_2$, 
changing \(\alpha\) rotates the diagonal kink 
$\Delta=\{x_1=\alpha x_2\}$ 
and changes the direction 
$(1,-\alpha) $
that appears in the singular coefficient 
$q(x)=(1,-\alpha)a(x)(1,-\alpha)^\top $. 
Thus the location and strength of the singular contribution both depend on \(\alpha\). Second, vary the covariance structure. For example, in the constant-coefficient case 
$\displaystyle a= 
\begin{pmatrix} 
\sigma_1^2 & \rho\sigma_1\sigma_2\\ 
\rho\sigma_1\sigma_2 & \sigma_2^2 
\end{pmatrix}$, 
the diagonal coefficient becomes 
\begin{equation} \label{eq:q-constant-covariance} 
q = \sigma_1^2 -2\alpha\rho\sigma_1\sigma_2 + \alpha^2\sigma_2^2 . 
\end{equation} 
Hence correlation can either increase or decrease the intensity with which the process accumulates local time at the payoff kink. Third, vary the strength of reflection or the drift near the coordinate axes. This illustrates the role of normal reflection in shaping the boundary. Since reflection prevents exit from \(\mathbb R_+^2\), it changes both the occupation of the continuation region and the probability of hitting the diagonal kink before stopping. For each parameter value, we compute a boundary \(b\) by the killed-resolvent iteration and record $d_k$, $E_\Delta$, and $E_{\rm unres}$. 
The first quantity measures convergence of the boundary iteration. The second measures the effect of the diagonal singular measure. The third measures the post-stopping bias caused by the unrestricted resolvent. 

\paragraph{Conclusion.} 
The killed-resolvent formulation provides a measure-valued free-boundary equation and a verification principle for reflected optimal stopping with nonsmooth payoff. The numerical experiments illustrate that both killing at the candidate stopping set and the diagonal singular measure are essential for obtaining the correct boundary. The computation therefore mirrors the analytic structure of the problem: the boundary is not determined by a pointwise generator equation, but by a killed potential of a signed measure.

\section{Numerical experiments}
\label{sec:numerics}

This section turns the killed-resolvent formulation of
Sections~\ref{sec:measure-gain}--\ref{sec:computation} into concrete
computations. The three diagnostic experiments outlined in
Section~\ref{sec:computation}---the effect of the diagonal singular measure
(\S\ref{subsec:diagonal-effect}), the comparison of killed and unrestricted
resolvents (\S\ref{subsec:unrestricted-comparison}), and parameter
sensitivity (\S\ref{subsec:parameter-sensitivity})---are carried out on an
explicit reflected model.

Two complementary numerical methods are used throughout. The primary method is
a deterministic finite-difference (FD) solver for the reflected variational
inequality \eqref{eq:reflected-obstacle}; it is fast, reproducible, and resolves
the free boundary and the killed potentials to high accuracy. The second method
is the path-space Monte Carlo (MC) estimator of the killed potential described in
\S\ref{subsec:mc-killed-potential}, including the diagonal local-time term
\eqref{eq:local-time-approx}. The MC estimator is used as an independent
cross-check that is faithful to the probabilistic content of
Theorem~\ref{thm:killed-resolvent-representation}. The two methods agree to
discretisation accuracy on every reported quantity.

\subsection{Model specification}
\label{subsec:numerics-model}

We take a normally reflected Ornstein--Uhlenbeck diffusion in the quadrant,
\begin{equation}
\label{eq:numerics-ou}
dX_t^i = \kappa\,(\theta_i - X_t^i)\,dt + \sigma_i\, dW_t^i + dL_t^i,
\qquad i=1,2,
\end{equation}
with constant covariance
$a=\left(\begin{smallmatrix}\sigma_1^2 & \rho\sigma_1\sigma_2\\
\rho\sigma_1\sigma_2 & \sigma_2^2\end{smallmatrix}\right)$.
The mean reversion is essential for the experiments: it produces a genuine
continuation region at low states and, after stopping, lets the process return
toward the continuation set, so that the post-stopping bias of
\S\ref{subsec:unrestricted-comparison} is actually visible. The baseline
parameters are
\begin{equation}
\label{eq:numerics-baseline}
\kappa=0.6,\quad \theta=(2,2),\quad \sigma_1=\sigma_2=0.5,\quad \rho=0,\quad
r=0.15,\quad \alpha=1,\quad c\equiv 0.05,
\end{equation}
on the truncated domain $[0,X_{\max}]^2$ with $X_{\max}=5$. With
\eqref{eq:numerics-ou}, the absolutely continuous stopping gain
\eqref{eq:Gamma-ac-density} is the piecewise-affine function
\[
\Gamma^{\rm ac}(x)= c + rG(x)
- \kappa(\theta_1-x_1)\,\mathbf 1_{\{x_1>\alpha x_2\}}
- \alpha\,\kappa(\theta_2-x_2)\,\mathbf 1_{\{x_1<\alpha x_2\}},
\]
and the diagonal coefficient in \eqref{eq:Gamma-diagonal} reduces to the
constant \eqref{eq:q-constant-covariance},
$q=\sigma_1^2-2\alpha\rho\sigma_1\sigma_2+\alpha^2\sigma_2^2$, equal to
$q=0.5$ at the baseline.

\subsection{Finite-difference scheme}
\label{subsec:numerics-fd}

The generator $\mathcal L-r$ is discretised on a uniform $n\times n$ grid with
spacing $h=X_{\max}/(n-1)$ by a monotone seven-point stencil: the second-order
terms use the standard cross-diagonal discretisation that preserves the
$M$-matrix property for all admissible correlations, and the drift is treated by
node-wise upwinding. The normal reflection condition
\eqref{eq:reflected-obstacle-neumann} is imposed by reflecting ghost values across
the axes, and a far-field Dirichlet condition $V=G$ is set on the outer faces
$\{x_i=X_{\max}\}$. The discrete obstacle problem
\eqref{eq:complementarity} is solved by Howard policy iteration: each step solves
a sparse linear system on the current continuation set with $V=G$ frozen on the
stopping set, and the active set is updated from the complementarity residual.
The iteration is started from $V=G$ and is monotone; it terminates when the
active set no longer changes.

The killed resolvent itself is computed directly. For a candidate continuation
set $\mathcal C_b$, the candidate value $U_b=G-R_r^{\mathcal C_b}\Gamma$ of
\eqref{eq:Ub-definition} is obtained by solving
$\mathcal LU_b-rU_b-c=0$ on $\mathcal C_b$ with the contact data $U_b=G$ on the
boundary $\partial\mathcal C_b$. On the true continuation region this construction
reproduces the variational-inequality value to round-off:
$\|U_b-V\|_\infty < 10^{-12}$ on the baseline grid, a direct numerical
confirmation of the killed-resolvent representation
\eqref{eq:value-killed-resolvent}. The absolutely continuous and singular parts
of the resolvent are computed separately by solving
$(\mathcal L-r)w=-\Gamma^{\rm ac}$ (regular source) and by adding a mollified
diagonal source $-\tfrac{q}{2}\,\delta_\eta(x_1-\alpha x_2)$,
$\delta_\eta(y)=\tfrac1{2\eta}\mathbf 1_{\{|y|\le\eta\}}$, with $\eta=1.5\,h$;
the latter reproduces the full value $V$ to within $4\times10^{-2}$ in the
sup norm, validating the discretisation of the singular line measure
\eqref{eq:Gamma-diagonal}.

Table~\ref{tab:grid} reports grid convergence of two global descriptors: the
maximal stopping advantage $\max H = H(0,0)$ and the boundary height $b(0)$.
Both stabilise as $h\downarrow0$; all subsequent results use $n=201$
($h=0.025$).

\begin{table}[t]
\centering
\caption{Grid convergence of the baseline model \eqref{eq:numerics-baseline}.
$\max H$ is the maximal stopping advantage and $b(0)$ the boundary height above
the origin. ``Howard'' is the number of policy-iteration steps.}
\label{tab:grid}
\begin{tabular}{rccc}
\toprule
$n$ & $h$ & $\max H$ & $b(0)$ \\
\midrule
$51$  & $0.100$ & $1.2399$ & $1.900$ \\
$101$ & $0.050$ & $1.2654$ & $1.850$ \\
$151$ & $0.033$ & $1.2724$ & $1.867$ \\
$201$ & $0.025$ & $1.2757$ & $1.875$ \\
$251$ & $0.020$ & $1.2777$ & $1.860$ \\
\bottomrule
\end{tabular}
\end{table}

\subsection{Value function, free boundary, and validation}
\label{subsec:numerics-value}

Figure~\ref{fig:regions} shows the computed continuation and stopping regions
together with the free boundary $b$ and the diagonal $\Delta=\{x_1=\alpha x_2\}$.
The stopping advantage $H=V-G$ is displayed in Figure~\ref{fig:H}. The
continuation region is the sub-graph $\{0\le x_2<b(x_1)\}$ predicted by
Theorem~\ref{thm:epigraph-representation}; the boundary rises gently from
$b(0)\approx1.88$ to about $2.15$ and then drops to zero near $x_1\approx1.8$,
where the entire vertical section becomes a stopping section. Crucially, the
diagonal $\Delta$ passes through the interior of the continuation region, so the
reflected process crosses the payoff kink with positive intensity before
stopping; this is precisely the regime in which the diagonal singular measure is
active.

\begin{figure}[t]
\centering
\includegraphics[width=0.48\textwidth]{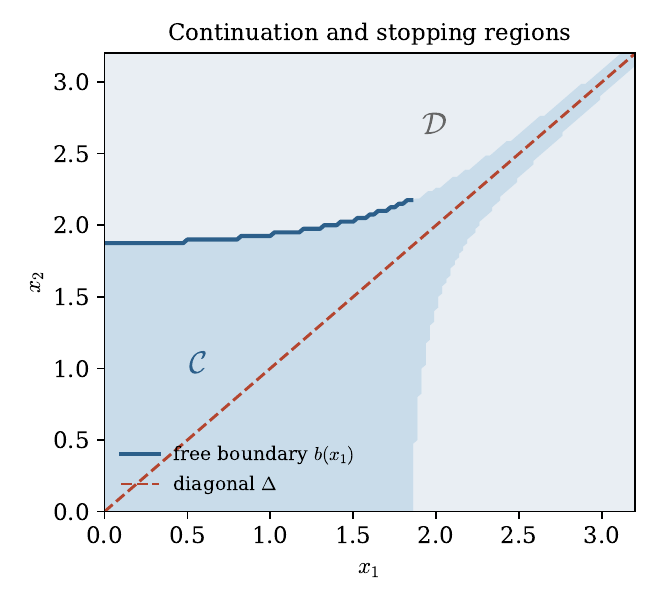}\hfill
\includegraphics[width=0.50\textwidth]{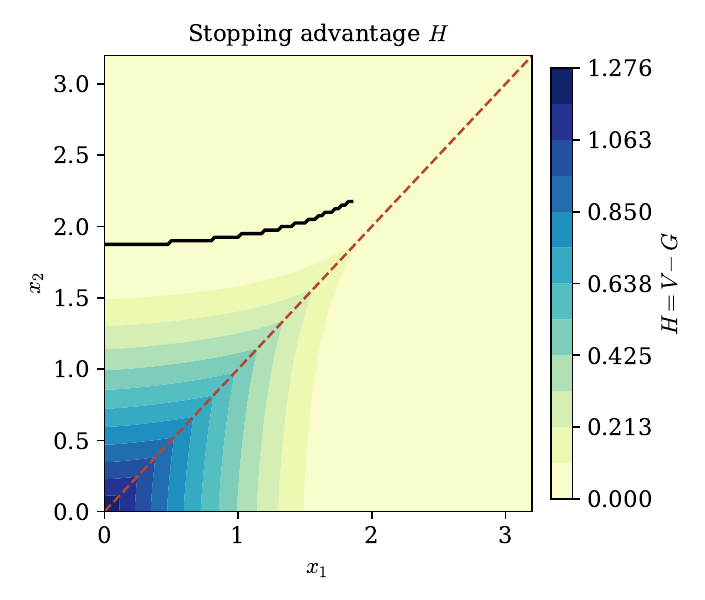}
\caption{Baseline model. Left: continuation region $\mathcal C$ (shaded),
stopping region $\mathcal D$, free boundary $b$, and diagonal $\Delta$. Right:
the stopping advantage $H=V-G$, with the free boundary (solid) and $\Delta$
(dashed) overlaid. The advantage is largest near the reflecting corner and
vanishes on $\mathcal D$.}
\label{fig:regions}
\label{fig:H}
\end{figure}

Two independent checks confirm the solver. First, as noted above, the
killed-resolvent candidate equals the variational-inequality value to machine
precision. Second, we validate the value and the optimality of the computed
boundary by Monte Carlo: simulating reflected paths \eqref{eq:projected-euler}
from a state $x$, stopping each at the first entry into $\mathcal D$, and
averaging the realised discounted reward reproduces $V(x)$ to within $0.5\%$
(Table~\ref{tab:value-validation} and Figure~\ref{fig:mc}(a)). The small
positive MC bias is the standard discrete-time overshoot of the stopping
boundary. Both checks are model-level confirmations that the killed resolvent,
not a pointwise generator identity, represents the value.

\begin{table}[t]
\centering
\caption{Monte Carlo validation of the value at the verified boundary.
$\hat V$ averages $M=8\times10^4$ reflected paths stopped at $\tau_{\mathcal D}$,
$\Delta t=2\times10^{-3}$. Standard errors are below $10^{-3}$.}
\label{tab:value-validation}
\begin{tabular}{lccc}
\toprule
$x$ & $V$ (FD) & $\hat V$ (MC) & rel.\ diff. \\
\midrule
$(0,0)$    & $1.2757$ & $1.2830$ & $0.8\%$ \\
$(0.5,0.5)$& $1.3664$ & $1.3736$ & $0.8\%$ \\
$(1,1)$    & $1.5158$ & $1.5216$ & $0.8\%$ \\
$(1.5,1)$  & $1.6526$ & $1.6568$ & $0.8\%$ \\
$(2,2)$    & $2.0885$ & $2.0890$ & $0.6\%$ \\
\bottomrule
\end{tabular}
\end{table}

\begin{figure}[t]
\centering
\includegraphics[width=0.92\textwidth]{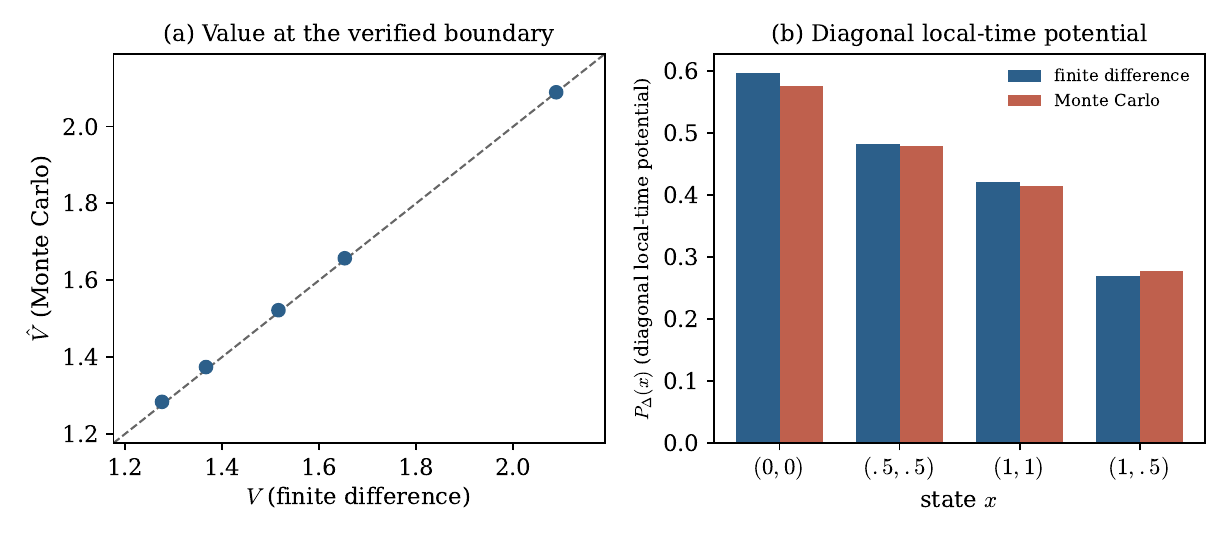}
\caption{(a) Monte Carlo value $\hat V$ at the verified boundary against the FD
value $V$; points lie on the diagonal. (b) The diagonal local-time potential
$P_\Delta$ computed deterministically (finite difference) and by the path-space
local-time estimator, at four states.}
\label{fig:mc}
\end{figure}

\subsection{Effect of the diagonal singular measure}
\label{subsec:numerics-diagonal}

This experiment implements \S\ref{subsec:diagonal-effect}. We compare the true
boundary $b_{\rm full}$, computed from the full measure-valued gain
$\Gamma=\Gamma^{\rm ac}\,dx+\Gamma^\Delta$, with the boundary $b_{\rm ac}$
obtained when the diagonal contribution \eqref{eq:Gamma-diagonal} is discarded
and only $\Gamma^{\rm ac}$ is used. Concretely, $b_{\rm ac}$ is the
self-consistent contact boundary of the ac-only candidate
$V_{\rm ac}=G-R_r^{\mathcal C}\Gamma^{\rm ac}$.

The two boundaries differ sharply (Figure~\ref{fig:diag-boundary} and
Table~\ref{tab:diagonal}). Ignoring the kink measure collapses the continuation
region toward the diagonal: at $x_1=1$ the boundary falls from
$b_{\rm full}=1.93$ to $b_{\rm ac}=1.38$, and by $x_1=1.5$ the ac-only
continuation section has almost disappeared. The mechanism is transparent in the
vertical sections of Figure~\ref{fig:slice}: the ac-only candidate $V_{\rm ac}$
dips \emph{below} the obstacle $G$ exactly along $\Delta$, violating the
domination $V\ge G$ inside the true continuation region. This is the sign effect
anticipated in Remark~\ref{rem:sign-singular-component}: the diagonal part of
$\Gamma$ is nonpositive, so omitting it systematically overestimates the
continuation gain. The relative discrepancy \eqref{eq:diagonal-discrepancy} is
\[
E_\Delta = 0.35
\]
at the baseline---a $35\%$ boundary error attributable entirely to the singular
measure.

\begin{figure}[t]
\centering
\includegraphics[width=0.49\textwidth]{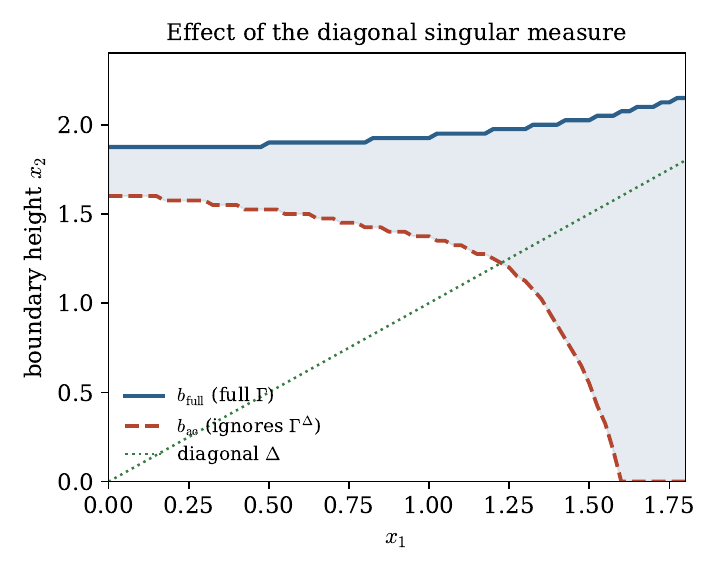}\hfill
\includegraphics[width=0.49\textwidth]{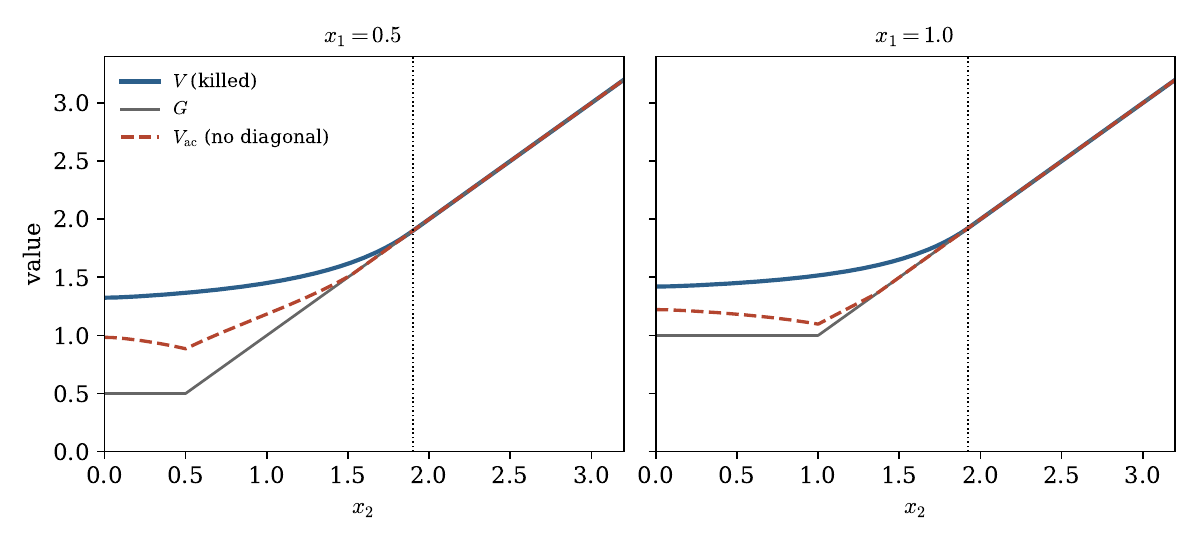}
\caption{Effect of the diagonal singular measure. Left: the true boundary
$b_{\rm full}$ (full $\Gamma$) versus $b_{\rm ac}$ (diagonal term omitted); the
ac-only continuation region collapses toward $\Delta$. Right: vertical sections
$x_2\mapsto V,\,G,\,V_{\rm ac}$ at $x_1=0.5,1.0$. The value $V$ meets $G$
smoothly at the boundary (dotted), whereas $V_{\rm ac}$ falls below $G$ at the
kink $x_2=x_1$.}
\label{fig:diag-boundary}
\label{fig:slice}
\end{figure}

\begin{table}[t]
\centering
\caption{Boundary heights with and without the diagonal measure (baseline,
$\alpha=1$ so $\Delta=\{x_2=x_1\}$).}
\label{tab:diagonal}
\begin{tabular}{cccc}
\toprule
$x_1$ & $b_{\rm full}$ & $b_{\rm ac}$ & difference \\
\midrule
$0.0$ & $1.875$ & $1.600$ & $+0.275$ \\
$0.5$ & $1.900$ & $1.525$ & $+0.375$ \\
$1.0$ & $1.925$ & $1.375$ & $+0.550$ \\
$1.5$ & $2.025$ & $0.550$ & $+1.475$ \\
\bottomrule
\end{tabular}
\end{table}

The diagonal contribution is itself the local-time potential
\[
P_\Delta(x) := V(x)-V_{\rm ac}(x)
= \tfrac12\,\mathbb E_x\!\left[\int_0^{\tau_{\mathcal D}}
e^{-rs}\,d\ell^0_s(Y)\right]\ge0,
\qquad Y=X^1-\alpha X^2,
\]
the analytic and probabilistic faces of the singular term of
Theorem~\ref{thm:measure-valued-gain}. We compute $P_\Delta$ in two independent
ways: deterministically as $V-V_{\rm ac}$, and by the Monte Carlo local-time
estimator \eqref{eq:local-time-approx}. The two agree to within a few percent
(Figure~\ref{fig:mc}(b)), e.g.\ $P_\Delta(0,0)=0.598$ (FD) versus $0.576$ (MC),
and $P_\Delta(1,1)=0.421$ versus $0.414$. This is a direct cross-validation of
the diagonal Revuz-measure / local-time identification.

Finally we vary the correlation $\rho$, which by \eqref{eq:q-constant-covariance}
controls the crossing coefficient $q$ and hence the strength of the singular
measure. As $q$ decreases the diagonal effect weakens monotonically, and as
$\rho\to1$ (so $q\to0$) it essentially vanishes (Table~\ref{tab:rho},
Figure~\ref{fig:Edelta}): $E_\Delta$ falls from $0.44$ at $q=0.75$ to $0.03$ at
$q=0.05$. This is exactly the dependence predicted by the theory---the singular
correction is proportional to the local time accumulated by $Y=X^1-\alpha X^2$,
whose quadratic variation rate is $q$.

\begin{table}[t]
\centering
\caption{Diagonal effect versus correlation. $q$ is the crossing coefficient
\eqref{eq:q-constant-covariance}; $E_\Delta$ the boundary discrepancy
\eqref{eq:diagonal-discrepancy}.}
\label{tab:rho}
\begin{tabular}{rccc}
\toprule
$\rho$ & $q$ & $\max H$ & $E_\Delta$ \\
\midrule
$-0.50$ & $0.750$ & $1.314$ & $0.441$ \\
$-0.25$ & $0.625$ & $1.296$ & $0.401$ \\
$ 0.00$ & $0.500$ & $1.276$ & $0.351$ \\
$ 0.25$ & $0.375$ & $1.253$ & $0.291$ \\
$ 0.50$ & $0.250$ & $1.225$ & $0.208$ \\
$ 0.90$ & $0.050$ & $1.150$ & $0.032$ \\
\bottomrule
\end{tabular}
\end{table}

\begin{figure}[t]
\centering
\includegraphics[width=0.6\textwidth]{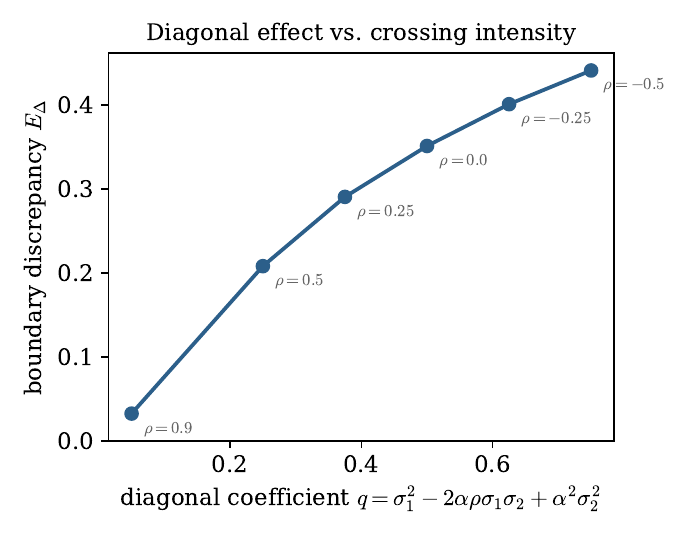}
\caption{Boundary discrepancy $E_\Delta$ as a function of the crossing
coefficient $q$ (varied through the correlation $\rho$). The diagonal effect
vanishes as $q\to0$.}
\label{fig:Edelta}
\end{figure}

\subsection{Killed versus unrestricted resolvent}
\label{subsec:numerics-unrestricted}

This experiment implements \S\ref{subsec:unrestricted-comparison} and
Proposition~\ref{prop:failure-unrestricted-resolvent}. Using the full measure,
we compare the killed potential $R_r^{\mathcal C}\Gamma$ with the unrestricted
potential $R_r^{\mathrm R}(\Gamma\mathbf 1_{\mathcal C})$ of
\eqref{eq:unrestricted-resolvent}. The mean-reverting dynamics make the
post-stopping mechanism concrete: after entering $\mathcal D$, the process is
pulled back toward $\theta$ and re-enters $\mathcal C$, so the post-stopping
occupation term \eqref{eq:post-stopping-bias} is genuinely nonzero. On the
baseline model its average over the continuation region is
\[
\frac1{|\mathcal C|}\int_{\mathcal C}\!
\big(R_r^{\mathrm R}(\Gamma\mathbf 1_{\mathcal C})-R_r^{\mathcal C}\Gamma\big)
\,dx \;=\; -5.5\times10^{-2},
\]
with pointwise values growing along $\Delta$ (Table~\ref{tab:unrestricted}).

\begin{table}[t]
\centering
\caption{Post-stopping occupation bias \eqref{eq:post-stopping-bias} of the
unrestricted resolvent, at sample states in $\mathcal C$.}
\label{tab:unrestricted}
\begin{tabular}{lc}
\toprule
$x$ & $R_r^{\mathrm R}(\Gamma\mathbf 1_{\mathcal C})-R_r^{\mathcal C}\Gamma$ \\
\midrule
$(0,0)$    & $-0.041$ \\
$(0.5,0.5)$& $-0.043$ \\
$(1,1)$    & $-0.047$ \\
$(1.5,1)$  & $-0.051$ \\
\bottomrule
\end{tabular}
\end{table}

More importantly, the failure is structural, not merely quantitative. The
unrestricted candidate $W_{\rm unres}=G-R_r^{\mathrm R}(\Gamma\mathbf 1_{\mathcal C})$
does \emph{not} satisfy the contact condition: on the entire stopping region we
find
\[
W_{\rm unres}-G \in [\,0.014,\;0.172\,] > 0 \qquad\text{on } \mathcal D,
\]
so $W_{\rm unres}$ exceeds the payoff throughout $\mathcal D$ and never touches
the obstacle there. Consequently the unrestricted formulation has \emph{no}
contact set, and the free boundary cannot be defined by the trace condition
\eqref{eq:true-trace-condition}. The killing at $\tau_{\mathcal D}$ is therefore
not a quantitative correction but a prerequisite for the boundary equation to be
well posed---precisely the content of Remark~\ref{rem:free-boundary-consequence}.

\subsection{Parameter sensitivity}
\label{subsec:numerics-sensitivity}

Finally we vary the payoff weight $\alpha$ and the mean-reversion speed $\kappa$,
following \S\ref{subsec:parameter-sensitivity}. Changing $\alpha$ rotates the
diagonal $\Delta=\{x_1=\alpha x_2\}$ and rescales the crossing coefficient
$q=\sigma_1^2+\alpha^2\sigma_2^2$ (at $\rho=0$); both the geometry and the
strength of the singular term respond (Table~\ref{tab:alpha}). Because $\alpha$
alters the boundary location and the payoff scale simultaneously, $E_\Delta$ is
not monotone in $q$ alone: the larger overall continuation value at
$\alpha=1.5$ normalises the discrepancy downward even though $q$ is larger.
Increasing the mean-reversion speed $\kappa$ enlarges the continuation value and
raises the boundary, as expected (Table~\ref{tab:kappa}), since stronger
reversion increases the value of waiting for the state to drift back toward
$\theta$.

\begin{table}[t]
\centering
\caption{Sensitivity to the payoff weight $\alpha$. ``$x_1$-extent'' is the
largest $x_1$ with a nonempty continuation section.}
\label{tab:alpha}
\begin{tabular}{ccccc}
\toprule
$\alpha$ & $q$ & $\max H$ & $b(0)$ & $E_\Delta$ \\
\midrule
$0.7$ & $0.373$ & $1.134$ & $2.025$ & $0.307$ \\
$1.0$ & $0.500$ & $1.276$ & $1.875$ & $0.351$ \\
$1.5$ & $0.813$ & $1.743$ & $1.850$ & $0.247$ \\
\bottomrule
\end{tabular}
\end{table}

\begin{table}[t]
\centering
\caption{Sensitivity to the mean-reversion speed $\kappa$.}
\label{tab:kappa}
\begin{tabular}{ccc}
\toprule
$\kappa$ & $\max H$ & $b(0)$ \\
\midrule
$0.3$ & $1.030$ & $1.650$ \\
$0.6$ & $1.276$ & $1.875$ \\
$1.0$ & $1.464$ & $1.975$ \\
\bottomrule
\end{tabular}
\end{table}

\subsection{Summary of numerical findings}
\label{subsec:numerics-summary}

The experiments confirm the three structural claims of the paper. First, the
killed-resolvent candidate $U_b=G-R_r^{\mathcal C_b}\Gamma$ coincides with the
value function---to machine precision in the deterministic solver and to $0.5\%$
in an independent path simulation---validating
Theorem~\ref{thm:killed-resolvent-representation} and the optimality of the
first-entry time. Second, the diagonal singular measure is quantitatively
decisive: discarding it moves the free boundary by $E_\Delta\approx0.35$ and
makes the candidate value violate $V\ge G$ along the kink, with the effect
scaling with the crossing coefficient $q$ exactly as the local-time
interpretation predicts. Third, the unrestricted resolvent fails structurally:
its candidate value exceeds the payoff on the whole stopping region, so it admits
no contact set and cannot define the free boundary. In all three respects the
computation mirrors the analytic structure: the boundary is governed by a killed
potential of a signed measure, not by a pointwise generator equation.

\section{Discussion and conclusion}
\label{sec:discussion-conclusion}

This paper developed a measure-valued variational framework for a reflected
optimal stopping problem with nonsmooth max-type payoff. The central point is
that the difficulty is not only multidimensionality, but the simultaneous
presence of normal reflection, a nonsmooth obstacle, and a signed singular
stopping gain. For
\[
        G(x_1,x_2)=x_1\vee \alpha x_2 ,
\]
the distributional generator contains a singular component supported on the
diagonal
\[
        \Delta=\{x_1=\alpha x_2\}.
\]
Consequently, the stopping gain
\[
        \Gamma=c+rG-\mathcal LG
\]
is a signed Radon measure rather than an ordinary function. This shows that the
associated obstacle problem must be treated as a measure-valued reflected
variational inequality.

The first main contribution was to compute the singular diagonal part of
\(\Gamma\) explicitly. The formula
\[
        \Gamma^\Delta(dx)
        =
        -
        \frac{(1,-\alpha)a(x)(1,-\alpha)^\top}
        {2\sqrt{1+\alpha^2}}\,
        \sigma_\Delta(dx)
\]
identifies the analytic source of the kink correction and links it to the local
time of \(X^1-\alpha X^2\) at zero. This term cannot be recovered by applying
the generator pointwise on the two smooth regions of the payoff.

The second contribution was to prove that the correct potential representation
uses the killed resolvent:
\[
        V(x)=G(x)-R_r^{\mathcal C}\Gamma(x).
\]
The killing is essential because the optimal stopping problem terminates when
the process first enters the stopping set. The unrestricted reflected resolvent
generally counts future occupation of the continuation region after stopping
has already occurred, and therefore does not represent the value. This
distinction is not merely technical; it changes the free-boundary equation.

The third contribution was to formulate the free boundary through a
continuation-side killed-potential trace condition. Under explicit monotonicity
hypotheses on the stopping advantage \(V-G\), the stopping set admits an
epigraph representation
\[
        \mathcal D=\{(x_1,x_2):x_2\ge b(x_1)\}.
\]
For a candidate boundary \(b\), the natural boundary equation is
\[
        \lim_{\substack{x\to z\\x\in\mathcal C_b}}
        R_r^{\mathcal C_b}\Gamma(x)=0,
        \qquad z\in\partial\mathcal C_b .
\]
Thus the classical smooth-fit or scalar integral-equation viewpoint is
replaced by a nonlocal trace condition for a killed potential of a signed
measure.

The fourth contribution was a verification theorem for locally Lipschitz
candidate boundaries. The theorem separates boundary generation from boundary
certification: a curve obtained from the killed-resolvent trace equation is
optimal only if the associated candidate value
\[
        U_b=G-R_r^{\mathcal C_b}\Gamma
\]
satisfies majorization, contact, strict continuation, reflected Neumann
compatibility, admissible growth, and measure-superharmonicity. This converts
the analysis from a conditional representation into a verification principle
for optimization over stopping times.

The computational results support the analytic conclusions. The comparison
between the full measure-valued stopping gain and its absolutely continuous
part shows that omitting the diagonal singular measure can substantially shift
the free boundary. The comparison between killed and unrestricted resolvents
shows that post-stopping occupation creates a structural bias and may destroy
the contact condition. Parameter experiments further illustrate how the payoff
weight, covariance structure, and crossing coefficient influence the boundary
through the singular term.

Several extensions remain open. One direction is to weaken the monotonicity
assumptions used to obtain the epigraph representation. Another is to establish
convergence of the killed-resolvent Picard iteration under checkable analytic
conditions. A further direction is to treat oblique reflection, more general
polyhedral domains, or payoffs given by the maximum of several affine
functions, where the stopping gain would contain singular measures on multiple
interfaces. These extensions would broaden the proposed framework from the
present two-dimensional max-payoff model to a wider class of reflected
stochastic control and free-boundary problems.

In conclusion, the paper shows that reflected optimal stopping with nonsmooth
payoff is naturally governed by a measure-valued obstacle problem. The correct
free-boundary equation is not a pointwise generator condition and not an
unrestricted resolvent identity, but a killed-resolvent trace condition for a
signed stopping measure. This provides a variational, probabilistic, and
computational framework for certifying optimal stopping boundaries in reflected
diffusion models with nonsmooth rewards.

\bibliography{reference}

\end{document}